\documentclass[3p,times]{elsarticle}
\usepackage{hyperref}
\usepackage{amsmath,amssymb,ntheorem, pifont}
\usepackage{graphicx}
\usepackage{caption}
\usepackage {setspace,bbding}
\usepackage{tabularx}
\usepackage{cleveref}
\usepackage{amssymb}
\usepackage[figuresright]{rotating}
\usepackage{algorithmic,algorithm}

\begin{document}

\begin{frontmatter}
\title{Rate Analysis of Coupled Distributed Stochastic Approximation for Misspecified Optimization \tnoteref{t1}}
\tnotetext[t1]{The paper was sponsored by the National Key Research and Development Program of China under No 2022YFA1004701,  the National Natural Science Foundation of China under No. 72271187 and No. 62373283, and partially by Shanghai Municipal Science and Technology Major Project No. 2021SHZDZX0100, and National Natural Science Foundation of China (Grant No. 62088101).}
\author[a1]{Yaqun Yang} 
\address[a1]{Department of Control Science and Engineering, Tongji University, Shanghai, 201804, China}

\author[a2]{Jinlong Lei\corref{mycorrespondingauthor}} 
\cortext[mycorrespondingauthor]{Corresponding author}
\address[a2]{Department of Control Science and Engineering, Tongji University, Shanghai, 201804, China; Shanghai Institute of Intelligent Science and Technology, Tongji University, Shanghai, 200092, China}

\ead{yangyaqun@tongji.edu.cn,  leijinlong@tongji.edu.cn}
\begin{abstract}
	    \qquad We consider an $n$ agents distributed optimization problem with imperfect information characterized in a parametric sense, where the unknown parameter  can  be solved by  a distinct distributed  parameter learning problem. Though each agent only has access to its local parameter learning and computational problem, they mean to collaboratively minimize the average of their local cost functions. 
To address the special optimization problem, we propose a coupled distributed stochastic approximation algorithm, in which every agent updates the current beliefs of its unknown parameter and decision variable by stochastic approximation method; then  averages the beliefs and decision variables of its neighbors over network in consensus protocol. Our interest lies in the convergence analysis of this algorithm. We quantitatively  characterize the factors that affect the algorithm performance, and prove that the  mean-squared error of the decision variable is  bounded by $\mathcal{O}(\frac{1}{nk})+\mathcal{O}\left(\frac{1}{\sqrt{n}(1-\rho_w)}\right)\frac{1}{k^{1.5}}
+\mathcal{O}\big(\frac{1}{(1-\rho_w)^2} \big)\frac{1}{k^2}$, where $k$ is the iteration count and $(1-\rho_w)$ is the spectral gap of the network weighted adjacency matrix. It reveals that the network connectivity characterized by $(1-\rho_w)$ only influences the high order of convergence rate, while the domain rate  still acts the same as the centralized algorithm. In addition, we analyze  that the transient iteration needed for reaching its dominant rate $\mathcal{O}(\frac{1}{nk})$ is $\mathcal{O}(\frac{n}{(1-\rho_w)^2})$. Numerical experiments are carried out to demonstrate the theoretical results by taking different CPUs as agents, which is more applicable to real-world distributed scenarios.
\end{abstract}


\end{frontmatter}
\section{Introduction}
\setlength{\jot}{-0.0005ex} 
In recent years, distributed optimization has drawn much research attention in various fields including economic dispatch\cite{dis_economic1,dis_economic2}, smart grids \cite{dis_grids, dis_cont_power, smartgrid1},  automatic controls\cite{dis_control, dis_auto2, dis_auto3} and machine learning \cite{ML1, ML2}. 
In distributed scenarios, each agent only preserves its local information, while they can exchange information with others over a connected network to cooperatively minimize the average of all agents' cost functions \cite{dis_frame,survey2}.  
There are several approaches for resolving  distributed optimization  problems such as (primary) consensus-based, duality-based, and constraint exchange methods, where  the primal approaches characterized by gradient-based algorithms have attracted the most research attention due to their satisfactory performance and well-scalable nature\cite{disgrad}.
The distributed dual approaches based on Lagrange method  regularly use  dual decomposition like the alternating direction method of multipliers (ADMM)\cite{disadmm}. Constraint exchange method is another prevalent scheme where the information exchanged by agents amounts to constraints\cite{disconstraint}.

However, among various formulations in distributed optimization, a crucial assumption is that we need precise objective functions (or problem information), i.e., all parameters in the model are precisely known. Yet in many economic and engineering problems, parameters of the functions are unknown but we may have access to observations that can aid in resolving this misspecification.   For example, in the Markowitz profile problem, it is routinely assumed that the expectation or covariance matrices associated with a collection of stocks are accurately available, but in reality, it needs empirical estimates via past data\cite{mislagrange}.

This paper is devoted to proposing distributed algorithms for resolving optimization problems with parametric misspecification, and quantitatively characterizing the influence of network properties, the heterogeneity of agents, initial states, etc. on the algorithm performance. This work is primarily centered around conducting a comprehensive theoretical analysis of convergence.  We begin by initiating the problem formulation.
\subsection{ Problem Formulation}\label{pro_form}
In this article, we consider a static misspecified distributed optimization problem defined as follows:
\begin{equation}
	\mathcal{C}_x(\theta_*): \qquad \min_{x\in\mathbb{R}^p}f(x,\theta_*)= \frac{1}{n}\sum\nolimits_{i=1}^{n}f_i(x,\theta_*)\label{0.1},
\end{equation}
where $f_i(x,\theta_*)\triangleq \mathbb{E}[\tilde{f}_i(x,\theta,\xi_i)] $ is the local cost function only accessible for agent  $i\in \mathcal{N} \triangleq \{1,2,...,n\}$.
Suppose that for any $i\in \mathcal{N} $, $\xi_i: \Omega_x \rightarrow \mathbb{R}^d$ are mutually independent random variables  defined on a probability space $(\Omega_x, \mathcal{F}_x, \mathbb{P}_x)$. Here, $\theta_*\in \mathbb{R}^q$ denotes the unknown parameter, which is a solution to a distinct convex problem.
\begin{equation}
	\mathcal{L}_{\theta}: \qquad \min_{\theta\in \mathbb{R}^q} h(\theta)=\frac{1}{n}\sum\nolimits_{i=1}^{n}h_i(\theta)\label{0.2},
\end{equation}
where $h_i(\theta)\triangleq\mathbb{E}[\tilde{h}_i(\theta,\zeta_i)]$ is the local parameter learning function only accessible for agent $i\in \mathcal{N} $, and for any $i \in\mathcal{N}$, $\zeta_i:\Omega_{\theta}\rightarrow\mathbb{R}^m$ are mutually independent random variables defined on a probability space $(\Omega_{\theta},\mathcal{F}_{\theta}, \mathbb{P}_{\theta})$.
Problems in the form \cref{0.1} and \cref{0.2} jointly formulate an unknown coupled  distributed optimization scheme consisting of both  {\itshape computational problem} and {\itshape learning problem}, where the learning problem is independent of the computational one. We have depicted the problem setting  in \cref{fig:problem set up}. 
\begin{figure}[htbp]
	\centering
	\includegraphics[width=100mm]{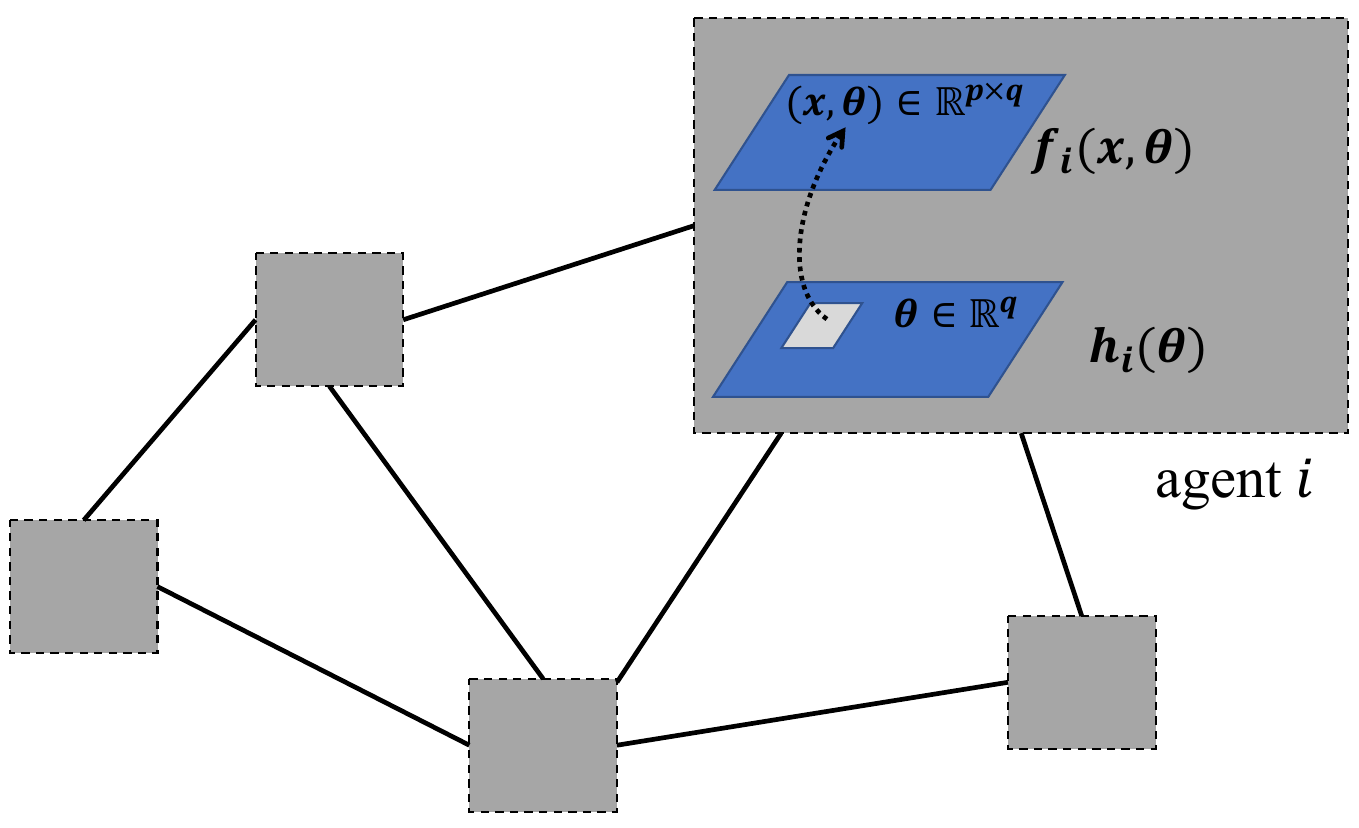}
	\caption{The problem setup: a connected network of communicating agents, where  each agent preserving a local learning problem $h_i$ and computational problem $f_i$ correlated with $h_i$ through the unknown parameter $\theta$, while they cooperate to solve the distributed coupled optimization problem.}
	\setlength{\belowcaptionskip}{-7cm}
	\label{fig:problem set up}
\end{figure}
\vspace{-12pt}
\subsection{Prior Work}
We now give a review of prior work for resolving optimization problems with unknown parameters.

{\it Robust optimization approach.}  Robust optimization considers the optimization problem where the parameter $\theta$ is unavailable but one can have access to its uncertainty set, say $\mathcal{U}_{\theta}$\cite{bertsimas2011RO}. The key idea is to optimize against the worst-case realization within this set, i.e.,
$$
\min_{x\in\mathbb{R}^p}\max_{\theta\in \mathcal{U}_{\theta}} f(x,\theta).$$
Robust optimization is shown to be a useful technique in the resolution of problems arising from  control, design, and optimization \cite{ben2009robust}. However, it usually produces conservative solutions and sometimes is   intractable  when poor set $\mathcal{U}_{\theta}$ is chosen (e.g. the set is given by unexplicit systems of non-convex inequalities)\cite{robustop}. 

{\it Stochastic optimization.} Unlike robust optimization, in a stochastic optimization scenario one may obtain statistical or distributional information about the unknown parameter. For example,  $\theta$ follows a probability distribution $\mathcal{D}$\cite{jie2018stochastic}, the optimal solution is gained by minimizing the expectation of cost functions,
\[
\min_{x \in \mathbb{R}^p} \mathbb{E}_{\theta\sim \mathcal{D}}[f(x,\theta)].
\]

Stochastic optimization has been widely investigated in telecommunication, finance and machine learning \cite{shapiro2021sto_prog}. In the scenario of a multi-agent network dealing with large datasets, stochastic optimization has become popular since it is challenging to calculate the exact gradient while the stochastic gradient is much easier to obtain. A key shortcoming in using  stochastic optimization models for resolving optimization problems with unknown parameters lies in that it needs the distribution of $\theta$, which might be a stringent requirement when the available data for estimating is limited or noisy. In such cases, the resulting distribution estimates may be unreliable or biased, leading to suboptimal solutions or even infeasible solutions\cite{wilson2018adaptive}. Alternatively, suppose that $\theta_*$ can be learnt by a suitably defined estimation  problem, then it brings about  the following approach.

{\it Data-driven learning approach.}
As data availability reaches hitherto unseen in recent years, we can use data-driven approaches to lessen or even eliminate the impact of  model uncertainty. For example, the model parameter  $\theta$ can be obtained by solving a suitably defined learning problem $l(\theta)$ (see e.g., \cite{jiang2016imperfect}),
\begin{equation}
	\min_{x \in \mathbb{R}^p} \left \{f(x,\theta_*): \theta_*\in \mathop{\arg\min}\limits_{\theta\in \mathbb{R}^q} l(\theta) \right \}.\label{cen_un}
\end{equation}
Computational evidence in portfolio management and queueing confirm that data-driven sets significantly outperform traditional robust optimization techniques\cite{robustop}.

A natural question is whether this problem could be solved in a sequential method, i.e., first accomplish estimating $\theta_*$ with high accuracy and then solve the core computational optimization problem with the achieved estimation $\hat{\theta}$. However, they have some disadvantages discussed in \cite{jiang2016imperfect, oco, miscen}: on the one hand, the large-scale parameter learning problem will lead to long time waiting for solving the original problem.  On the other hand, this scheme provides an approximate solution $\hat{\theta}$, then the corrupt error might propagates into the computational problem. As such, sequential methods cannot provide asymptotically accurate convergence.
Therefore, an alternative simultaneous approach   is designed (see e.g.,  \cite{jiang2016imperfect,guolei}), which use observations to get an estimation  $\theta_k$ of unknown parameters $\theta^*$ at each time instant $k$;  then update the upper optimization problem by taking the estimated parameter $\theta_k$  as ``true" parameter. This simultaneous approach can generate  a sequence $\{(x_k,\theta_k)\}$ that converges to a minimizer of $f(x,\theta_*)$ and $l(\theta)$ respectively \cite{jiang2016imperfect}.

Such data-driven learning approaches for unknown parameter has gradually attracted  research attention recently. For example, the authors of
%
\cite{jiang2016imperfect} presented a centralized coupled stochastic optimization scheme to solve problem \eqref{cen_un} and showed the convergence properties in regimes when the function is either strongly convex or merely convex. Then \cite{miscen} extended it smooth or nonsmooth functions $f$ and presented an averaging-based subgradient approach, but it is still a centralized scheme. In addition, the authors of \cite{oco} divided the optimization problem with uncertainty into two paradigms: robust optimization and joint estimation optimization, and they exploited these  two problem structures in online convex optimization and gave regret analysis under different conditions. The  recent work \cite{mislagrange} investigated the  misspecified conic convex programs, and developed  a centralized first-order inexact augmented Lagrangian scheme for computing the optimal solution while simultaneously learning the unknown parameters. The aforementioned work \cite{mislagrange,oco,jiang2016imperfect,miscen} all investigated  centralized methods, while there are some other work exploit distributed approaches. For example,   \cite{distributedmis} considered the distributed stochastic optimization with imperfect information, while it only showed that the generated iterates converge almost surely to the optimal solution.
Though the work \cite{personalizedgra_traking} presented a distributed problem with a composite structure consisting of an exact engineering part and an unknown personalized part, 
it exhibits a bounded regret under certain conditions. 
\subsection{Gaps and Motivation}

Recalling the problem setup in \cref{pro_form}, our research falls into distributed data-driven stochastic optimization scenario. Taking into account the research that is most pertinent to this paper, the majority of previous studies have primarily concentrated on centralized inquiries (see e.g. \cite{mislagrange,oco,jiang2016imperfect,miscen}), 
while the distributed schemes \cite{distributedmis,personalizedgra_traking} mainly investigated the asymptotic convergence. It remains unknown how to design an efficient distributed algorithm, how does the network connectivity influence the algorithm performance, and whether the rate can reach the same order as the centralized scheme? To be specific, this paper is motivated by the following gaps: (i) previous work on unknown parameter learning problems focused on the centralized scheme, the distributed data-driven stochastic approximation method is less studied; (ii) the discussion of convergence analysis especially how factors such as the number of agents, the network  connectivity, and the heterogeneity of agents influence the rate of algorithm is rarely studied in details;  (iii) the gap between centralized and distributed algorithm under imperfect information need to be specified, or in other words  can we find the transient  time when the rate of distributed algorithms reach  the same  order as that of the centralized scheme?

\subsection{ Outline and Contributions}  To address these gaps, we  propose a data-driven coupled distributed stochastic approximation method to resolve this special optimization problem and give a precise convergence rate analysis of our algorithm. The main contributions are summarized as follows, and the comparison with previous works is shown in \cref{com}.
\begin{table}[tbhp]
	\footnotesize
	\vspace{-5pt}
	\caption{Work comparation with previous studies}
	\label{com}
	\vspace{-5pt}
	\begin{center}
		\begin{tabular}{|c|c|c|c|c|c|}
			\hline
			Paper & Distributed & Imperctect Information  & Stochastic & Rate & Factor Influence\\
			\hline
			\cite{oco,jiang2016imperfect,miscen} & \XSolidBrush & \Checkmark & \XSolidBrush & $\mathcal{O}\left(\frac{1}{k}\right)$ & \XSolidBrush \\
			\hline
			\cite{du2022computational,pu2021sharp} &\Checkmark  &  \XSolidBrush & \Checkmark & $\mathcal{O}\left(\frac{1}{k}\right)$ &\Checkmark \\
			\hline
			\cite{personalizedgra_traking} & \Checkmark  & \Checkmark  & \XSolidBrush &$\backslash$ & \Checkmark \\
			\hline
			\cite{distributedmis} & \Checkmark & \Checkmark & \Checkmark &$\backslash$ & \XSolidBrush\\
			\hline
			\cite{liang2019distributed} & \Checkmark & \XSolidBrush & \XSolidBrush &$\mathcal{O}\left(\frac{1}{\sqrt{k}}\right)$ & \XSolidBrush\\
			\hline
			Our Work & \Checkmark & \Checkmark & \Checkmark & $\mathcal{O}\left(\frac{1}{k}\right)$ & \Checkmark\\
			\hline
		\end{tabular}
	\end{center}
\end{table}
\vspace{-2pt}
%

(1) We propose a coupled distributed stochastic approximation algorithm that generates iterates \{$(\pmb{\mathbf{x}}(k), \pmb{\mathbf{\theta}}(k))$\} for the distributed stochastic optimization  problem \eqref{0.1}  with the  unknown parameter learning prescribed by  a separate distributed stochastic optimization  problem  \eqref{0.2}. Our model framework builds upon previous research involving deterministic and stochastic gradient schemes. This is particularly relevant for certain studies where waiting for parameter learning to complete over an extended period is not feasible, or for real-world problems in which parameter learning and objective optimization are intertwined.

(2) We characterize the  convergence rate of the presented algorithm that  combined the distributed consensus protocol with stochastic gradient descent methods. On the one hand, we prove that the upper bound of expected consensus error for every agent decay at rate $\mathcal{O}(\frac{1}{k^2})$; on the other hand, we also show that the upper bounded of expected optimization error is $\mathcal{O}(\frac{1}{k})$. We then give the sublinear convergence rate   and quantitatively characterize some factors affecting the convergence rate, such as the network size, spectral gap of the weighted adjacency matrix, heterogenous of individual function, and  initial values. We emphasize   that the  mean-squared error of the decision variable is  bounded by $\mathcal{O}(\frac{1}{nk})+\mathcal{O}\left(\frac{1}{\sqrt{n}(1-\rho_w)}\right)\frac{1}{k^{1.5}}
+\mathcal{O}\big(\frac{1}{(1-\rho_w)^2} \big)\frac{1}{k^2}$, which  indicates that the network connectivity characterized by $(1-\rho_w)$ only influences the high order of convergence rate, while the domain rate $\mathcal{O}(\frac{1}{ k})$ still acts the same as the centralized algorithm.

(3) We analyze the transient time $K_T$ for the proposed algorithm, namely, the number of iterations before the algorithm reaches its dominant rate.  Specially, we show that when the iterate $k\geq K_T$, the dominant factor influencing the convergence rate is related to stochastic gradient descent,
while for small $k<K_T$, the main factor influencing the convergence rate originates from the distributed average consensus method. Finally, we show that the algorithm asymptotically achieves the same network-independent convergence rate as the centralized scheme.

The paper is organized as follows. We present the algorithm and the related assumptions in \cref{sec:algorithm}. In \cref{sec:Aux}, the auxiliary results supporting the convergence rate analysis is proved. Our main results are in \cref{sec:mainreslut}. Experimental
results are  implemented in \cref{sec:experiments}, while the concluding remarks are given  in \cref{sec:conclusions}.

\textbf{Notation. }
All vectors in this paper are column vectors. 
The structure of the communication network is modeled by an undirected weighted graph $\mathcal{G} = (\mathcal{N},\mathcal{E},\mathcal{W})$ in which $\mathcal{N}=\{1,2,...,n\}$ represents the set of vertices. $\mathcal{E}\subseteq \mathcal{N}\times\mathcal{N}$ is the set of edges.  $W=[w_{ij}]_{n\times n}\in \mathbb{R}^{n\times  n}$ denotes the weighted adjacency matrix, $w_{ij}> 0$ if and only if agent $i$ and agent $j$ are connected, $w_{ij}=w_{ji}=0$ otherwise. Each agent(vertice) has a set of neighbors $\mathcal{N}_i=\{j|(i,j)\in\mathcal{E}\}$. The graph is connected means for every pair of nodes $(i,j)$ there exists a path of edges that goes from $i$ to $j$. $||\cdot||$ denotes $\mathcal{L}_2$-norm for vectors and Euclidean norm for matrices. The optimal solution denote as ($x_*, \theta_*$).
\vspace{-10pt}
\section{Algorithm and Assumptions}
\label{sec:algorithm}
To solve this special optimization problem consisting of the {\itshape computational problem} \cref{0.1} and the {\itshape learning problem} \cref{0.2}, we will propose a {\itshape Coupled Distributed Stochastic Approximation (CDSA) Algorithm}  and impose some conditions for   rate analysis in this section.
\subsection{Algorithm Set Up}
As mentioned previously, each agent $i$ only knows its local core computational function $f_i(x,\theta)$ and parameter learning function $h_i(\theta)$, while they are connected by a network $\mathcal{G} = (\mathcal{N},\mathcal{E} ,\mathcal{W})$ in which agents may communicate and exchange information with their neighbors $\mathcal{N}_i=\{j|(i,j)\in\mathcal{E}\}$. 
At each step $k\geq 0$, every agent $i$ holds an estimate of the decision variable and unknown parameter, denoted by $x_i(k)$ and $\theta_i(k)$, respectively. Suppose that every agent has access to a stochastic first-order oracle  that can generate  stochastic gradients $g_i(x_i(k),\theta_i(k),\xi_i(k))\triangleq\nabla_x f_i(x_i(k),\theta_i(k),\xi_i(k))$ and $\phi_i(\theta_i(k),\zeta_i(k))\triangleq\nabla_{\theta} h_i(\theta_i(k),\zeta_i(k))$ respectively (where $\xi_i, \zeta_i, i=1,2,...,n$ are independent random variables). Then,  every agent updates its parameters through stochastic gradient descent method to obtain temporary  estimates $\widetilde{x}_i(k)$ and $\widetilde{\theta}_i(k)$. Next, each agent communicates with its local neighbors and gathers temporary parameters information over a static connected network to renew the iterates $x_i(k+1)$ and  $\theta_i(k+1)$ based on the consensus protocol. We summarize the  pseudo-code is in \cref{alg:CDSA}.
\begin{algorithm}
	\caption{Coupled Distributed Stochastic Approximation (CDSA)}
	\label{alg:CDSA}
	\begin{algorithmic}
		\STATE{\textbf{Initialization:} $W=[w_{ij}]_{n\times n}; (x_i(0),\theta_i(0)), \forall i\in \mathcal{N}$}
		\STATE{\textbf{Evolution:} for $k=0,1,2,...; \forall i\in \mathcal{N}$}
		\STATE{\textbf{\qquad Compute:} stochastic gradient $\phi_i\left(\theta_i(k),\zeta_i(k)\right)$ and $g_i(x_i(l),\theta_i(k),\xi_i(k))$}
		\STATE{\textbf{\qquad Choose:}} stepsize $\alpha_k$ and $\gamma_k$ (To be introduced in \cref{uni.bound_stepsize})
		\STATE{\textbf{\qquad Update } according to the following stochastic gradient descent method.}
		\begin{align}
			\widetilde{x}_i(k)&=x_i(k)-\alpha_kg_i(x_i(k),\theta_i(k),\xi_i(k))\notag\\
			\widetilde{\theta}_i(k)&=\theta_i(k)-\gamma_k\phi_i\left(\theta_i(k),\zeta_i(k)\right)\notag		
		\end{align}
		\STATE{\textbf{\qquad Gather}} information $\widetilde{x}_j(k),\widetilde{\theta}_j(k)$ from its neighbors $j\in \mathcal{N}_i$
		and renew the iterates by the consensus protocol below.
		\begin{align}		
			x_i(k+1)&=\sum\nolimits_{j\in\mathcal{N}_i}w_{ij}\widetilde{x}_j(k)\notag\\	\theta_i(k+1)&=\sum\nolimits_{j\in\mathcal{N}_i}w_{ij}\widetilde{\theta}_j(k)\notag
		\end{align}
	\end{algorithmic}
\end{algorithm}

We can rewrite Algorithm \ref{alg:CDSA} in a more compact form as follows.
\begin{align}
	x_i(k+1)&=\sum\nolimits_{j\in\mathcal{N}_i}w_{ij}(x_j(k)-\alpha_k g_j(x_i(l),\theta_i(k),\xi_i(k))),\label{3.1}\\
	\theta_i(k+1)&=\sum\nolimits_{j\in\mathcal{N}_i}w_{ij}\left(\theta_j(k)-\gamma_k\phi_j\left(\theta_i(k),\zeta_i(k)\right)\right).\label{3.2}
\end{align}
Define
\begin{align}
	\pmb{\mathbf{x}} & \triangleq [x_1,x_2,\cdots,x_n]^T\in \mathbb{R}^{n\times p},
	\pmb{\mathbf{\theta}}  \triangleq [\theta_1,\theta_2,\cdots,\theta_n]^T\in \mathbb{R}^{n\times q},\label{theta_vec}\\%
	\boldsymbol{\xi}&\triangleq [\xi_1,\xi_2,\cdots,\xi_n]^T\in \mathbb{R}^n,~\boldsymbol{\zeta}\triangleq [\zeta_1,\zeta_2,\cdots,\zeta_n]^T\in \mathbb{R}^n,\\
	\mathbf{g}(\pmb{\mathbf{x}},\pmb{\theta},\pmb{\mathbf{\xi}})&\triangleq[g_1(x_1,\theta_1,\xi_1),g_2(x_2,\theta_2,\xi_2),\cdots,g_n(x_n,\theta_n,\xi_n)]^T\in \mathbb{R}^{n\times p}\label{g_vec},\\%
	\boldsymbol{\phi}(\pmb{\theta},\pmb{\mathbf{\zeta}})& \triangleq [\phi_i\left(\theta_1,\zeta_1\right),\phi_i\left(\theta_2,\zeta_2\right),\cdots,\phi_i\left(\theta_n,\zeta_n\right)]^T\in \mathbb{R}^{n\times q}.\label{phi_vec}
\end{align}
Then equation (\ref{3.1}) and (\ref{3.2}) can be reformulated in the following vector formula.
\begin{align}
	\pmb{\mathbf{x}} (k+1)&=W\left(\pmb{\mathbf{x}}(k)-\alpha_k\mathbf{g}(\pmb{\mathbf{x}}(k),\pmb{\theta}(k),\pmb{\mathbf{\xi}}(k))\right),\label{vec_xiter}\\
	\pmb{\mathbf{\theta}} (k+1)&=W\left(\pmb{\mathbf{\theta}}(k)-\alpha_k\boldsymbol{\phi}(\pmb{\theta}(k),\pmb{\mathbf{\zeta}}(k))\right).\label{vec_thetaiter}
\end{align}
\vspace{-20pt}
\subsection{Assumptions}
In this subsection, we will specify the conditions for rate analysis of the CDSA algorithm.  We need to make some assumptions about the properties of objective functions in both learning and computation metrics to get the global optimal solution.  Besides,
we impose some constraints on conditional first and second moments of  ``stochastic gradient". Last but not least, we inherit the typical assumptions about communication networks as that of distributed algorithms.

\newtheorem{mythm}{Assumption}[subsection]
\begin{mythm}[Function properties]\label{assump. func_pro}
	(\romannumeral1) For every $\theta\in \mathbb{R}^q$, $f_i(x,\theta),i=1,\cdots,n$ is strongly convex and Lipschitz smooth in $x$ with constants $\mu_x$ and $L_x$, i.e.
	\begin{align*}
		(\nabla_x f_i(x',\theta)-\nabla_x f_i(x,\theta))^T(x'-x)&\geq \mu_x||x'-x||^2, \forall x, x'\in\mathbb{R}^p,\\
		||\nabla_x f_i(x',\theta)-\nabla_x f_i(x,\theta)||&\leq L_x||x'-x||,\forall x, x'\in \in\mathbb{R}^p.
	\end{align*}
	
	(\romannumeral2) For every $x\in \mathbb{R}^p$,  $f_i(x,\theta),i=1,\cdots,n$ is strongly convex and Lipschitz smooth in $\theta$ with constants $\mu_{\theta}$ and $L_\theta$ respectively, i.e.
	\begin{align*}
		(\nabla_x f_i(x,\theta')-\nabla_x f_i(x,\theta))^T(\theta'-\theta)&\geq \mu_{\theta}||\theta'-\theta||^2, \forall \theta,\theta'\in \mathbb{R}^q,\\
		||\nabla_x f_i(x,\theta')-\nabla_{\theta} f(x,\theta)||&\leq L_{\theta}||\theta'-\theta||, \forall \theta,\theta'\in \mathbb{R}^q.
	\end{align*}
	
	(\romannumeral3)The learning metric $h_i(\theta)$ for every $i\in \{1,2,...,n\}$ is strongly convex and Lipschitz smooth with constants $\nu_{\theta}$ and $C_{\theta}$, i.e.
	\begin{align*}
		(h(\theta)-h(\theta'))^T(\theta-\theta')&\geq \nu_{\theta}||\theta-\theta'||^2, \forall \theta,\theta'\in \mathbb{R}^q,\\
		||\nabla h(\theta)-\nabla h(\theta')||&\leq C_{\theta}||\theta-\theta'||, \forall \theta,\theta'\in \mathbb{R}^q.
	\end{align*}
\end{mythm}
Strong convexity assumptions indicate that both  computational problem and learning problem have a unique optimal solution $x_*\in \mathbb{R}^p$ and $\theta_*\in \mathbb{R}^q$ \cite{largeML}. The Lipschitz continuity  of  gradient functions ensure that the gradient doesn't change arbitrarily fast concerning the corresponding parameter vector. It is widely used in the convergence analyses of most gradient-based methods, without it, the gradient wouldn't provide a good indicator for how far to move to decrease the objective function\cite{largeML}.
These assumptions are satisfied for many machine learning problems, such as logistic regression, linear regression, and support vector machine (SVM).

Next, we define a new probability space $(\mathcal{Z},\mathcal{F},\mathbb{P})$, where $\mathcal{Z}\triangleq \Omega_x \times \Omega_{\theta}, \mathcal{F}\triangleq \mathcal{F}_x \times \mathcal{F}_{\theta}$ and $\mathbb{P}\triangleq \mathbb{P}_x\times \mathbb{P}_{\theta}$.
We use $\mathcal{F}(k)$ to denote the $\sigma$-algebra generated by $ \{(x_i(0),\theta_i(0)),(x_i(1),\theta_i(1)),...,(x_i(k),\theta_i(k))|i\in\mathcal{N}\}$. Then give the following assumptions related to the stochastic gradient estimator,
which assume that the stochastic gradient is an unbiased estimator of the true gradient, and the variance of the stochastic gradient is restricted.

\begin{mythm}[Conditional first and second moments]\label{assump. moment}
	For all $k\geq 0$ and $i\in \mathcal{N}$,there exist $\sigma_x>0, \sigma_{\theta}>0, M_x>0, M_{\theta}>0$, such that
	\begin{align*}
		&(a)\quad	\mathbb{E}_{\xi_i(k)}[g_i(x_i(k),\theta_i(k),\xi_i(k))|\mathcal{F}(k)]=\nabla_x f_i(x_i(k),\theta_i(k)), \quad a.s.,\\
		&(b)\quad	\mathbb{E}_{\zeta_i(k)}[\phi_i(\theta_i(k),\zeta_i(k))|\mathcal{F}(k)]=\nabla h_i(\theta_i(k)), \quad a.s.,\\
		&(c)\quad	\mathbb{E}_{\xi_i(k)}[||g_i(x_i(k),\theta_i(k),\xi_i(k))-\nabla_x f_i(x_i(k),\theta_i(k))	||^2|\mathcal{F}(k)], \notag\\
		&\qquad\qquad\qquad\qquad	\leq \sigma_x^2+M_x||\nabla_x f_i(x_i(k),\theta_i(k))||^2\quad a.s.,\\
		&(d)\quad	\mathbb{E}_{\zeta_i(k)}[||\phi_i(\theta_i(k),\zeta_i(k))-\nabla h_i(\theta_i(k))||^2|\mathcal{F}(k)]\leq \sigma_{\theta}^2+M_{\theta}||\nabla h_i(\theta_i(k))||^2	, \quad a.s.,
	\end{align*}
\end{mythm}

Next, we impose the connectivity condition on the graph, which  indicates that after multiple rounds of communication, information can be exchanged between any two agents.  This inherits the typical assumptions on consensus protocols \cite{fastdis_ave}.
\begin{mythm}[Graph and weighted matrix]\label{assump. graph}
	The graph $\mathcal{G}$ is static, undirected, and connected. The weighted adjacency matrix $W$ is nonnegative and doubly stochastic, i.e.,
	\begin{equation}
		W \boldsymbol{1}=\boldsymbol{1},\boldsymbol{1}^TW=\boldsymbol{1}^T\label{double stoch.matrix}
	\end{equation}
	where $\boldsymbol{1}$ is the vector of all ones.
\end{mythm}

Next, we state two lemmas that partially explain the practicability of  \cref{alg:CDSA}   based on the aforementioned  assumptions.

\newtheorem{lemma1}{Lemma}[subsection]
\begin{lemma1}\label{iteration gradient work}\cite[Lemma 10]{qu2017harnessing}
	For any $x\in\mathbb{R}^p$,	define $x^+=x-\alpha \nabla f(x)$. Suppose that  $f$ is strongly convex with constant $\mu$ and its gradient function is Lipschitz continuous with constant $L$. If $\alpha\in(0,2/L)$, we then have
	$
	||x^+-x_*||\leq \lambda||x-x_*||,$
	where $\lambda\triangleq\max (|1-\alpha\mu|,|1-\alpha L|)$.
\end{lemma1}

It can be observed from the above lemma that as long as we choose a proper stepsize ($0<\alpha<2/L$), the distance to optimizer shrinks by a ratio $\lambda<1$ at each step for strongly convex and smooth functions.
While the following lemma reveals that under distributed algorithm with linear iteration, the gap between the current iteration and consensus optimal solution is decreased by a ratio $\rho_w<1$ compared to the last iteration.
\begin{lemma1}\label{iteration consensus work} \cite[Theorem 1]{fastdis_ave}
	Let Assumption \ref{assump. graph} hold, and $\rho_w$ denote the spectural norm of matrix $W-\frac{\pmb{1}\pmb{1}^T}{n}$. Thus $\rho_w<1$. Define $\pmb{\omega}^+=W\pmb{\omega}$ for any $\pmb{\omega}\in\mathbb{R}^{n\times p}$. We then have $||\pmb{\omega}^+-\pmb{1}\bar{\omega}||\leq\rho_w||\pmb{\omega}-\pmb{1}\bar{\omega}||, $
	where $\bar{\omega} \triangleq \frac{1}{n}\pmb{1}^T\pmb{\omega}$.
\end{lemma1}

The aforementioned lemmas show that both the gradient descent method and distributed linear iteration can move  the decision variable  towards the optimal solution with
linear decaying rates. Thus, our algorithm consisting of both approaches might  find the optimal solution efficiently.
We will rigorously prove the convergence rate of \cref{alg:CDSA}  in the following two sections.

\vspace{-5pt}
\section{Auxiliary Results}
\label{sec:Aux}
In this section, we will present some results to assist subsequent convergence rate analysis. We first give some preliminary bound which will be used  for later proof, then present the supporting lemmas concerning recursions for expected optimization error and expected consensus error, and finally, we prove that under diminishing stepsize, the mean-squared distance   between the current iterate and the optimal solution is uniformly bounded.
\subsection{Preliminary Bound}
For simplicity, we denote
\begin{spacing}{1}
	\begin{align}
		\bar{x}(k)\triangleq\frac{1}{n}\sum\nolimits_{i=1}^{n}&x_i(k),~
		\bar{\theta}(k)\triangleq\frac{1}{n}\sum\nolimits_{i=1}^{n}\theta_i(k),\label{bar_xtheta}\\ \bar{g}(\pmb{\mathbf{x}}(k),\pmb{\mathbf{\theta}}(k),\pmb{\mathbf{\xi}}(k))&\triangleq\frac{1}{n}\sum\nolimits_{i=1}^{n}g_i(x_i(k),\theta_i(k),\xi_i(k))\label{bar_g},\\
		\bar{\nabla}_xF(\pmb{\mathbf{x}}(k),\pmb{\mathbf{\theta}}(k))&\triangleq \frac{1}{n} \sum\nolimits_{i=1}^n \nabla_x f_i\left(x_i(k) , \theta_i(k)\right)\label{bar_F}.
	\end{align}
\end{spacing}
We will show in the following lemma that  with Assumptions \ref{assump. func_pro} and  \ref{assump. moment}, the conditional squared  distance between the gradient $\bar{\nabla}_xF(\pmb{\mathbf{x}}(k),\pmb{\mathbf{\theta}}(k))$ and its estimate can be bounded by linear combinations of squared errors $||\pmb{\mathbf{x}}(k)-\pmb{1}x_*^T||^2$  and   $||\pmb{\mathbf{\theta}}(k)-\pmb{1}\theta_*^T||^2$. For completeness, its proof is given  in \cref{appA}.

\begin{lemma1}\label{55.3}
	Let Assumption \ref{assump. func_pro} and \ref{assump. moment} hold. Then for any $k\geq0$,
	\begin{align} \mathbb{E}[||\bar{g}(\pmb{\mathbf{x}}(k),\pmb{\mathbf{\theta}}(k),\pmb{\mathbf{\xi}}(k))&-\bar{\nabla}_xF(\pmb{\mathbf{x}}(k),\pmb{\mathbf{\theta}}(k))||^2|\mathcal{F}(k)] \notag\\
		&\leq \frac{3M_xL_x^2}{n^2}||\pmb{\mathbf{x}}(k)-\pmb{1}x_*^T||^2+\frac{3M_xL_{\theta}^2}{n^2}||\pmb{\mathbf{\theta}}(k)-\pmb{1}\theta_*^T||^2+\frac{\bar{M}}{n}\label{5.1}, \\
		\text{where}\quad		
		\bar{M}&=\frac{3M_x\sum\nolimits_{i=1}^{n}||\nabla_x f_i(x_*,\theta_*)||^2}{n}+\sigma_x^2\label{barM}.
	\end{align}
\end{lemma1}

The following lemma shows the gap between the gradient of objective function
at the consensual points $(\bar{x}(k),\bar{\theta}(k))$, denoted by $\frac{1}{n}\sum\nolimits_{i=1}^{n}\nabla_xf_i(\bar{x}(k),\bar{\theta}(k))$,  and at current iterates  $\frac{1}{n}\sum\nolimits_{i=1}^{n}\nabla_xf_i(x_i(k),\theta_i(k))$  can also be bounded by linear combinations of $||\pmb{\mathbf{x}}(k)-\pmb{1}x_*^T||^2$  and  $||\pmb{\mathbf{\theta}}(k)-\pmb{1}\theta_*^T||^2$. The precise proof is in \cref{appB}.

\begin{lemma1}\label{55.4}
	Let Assumption \ref{assump. func_pro} hold. Then for any $k\geq0$,
	\begin{equation}
		||\nabla_xf(\bar{x}(k),\bar{\theta}(k))-\bar{\nabla}_xF(\pmb{\mathbf{x}}(k),\pmb{\mathbf{\theta}}(k))||\leq\frac{L_x}{\sqrt{n}}||\pmb{\mathbf{x}}(k)-\pmb{1}\bar{x}(k)^T||+\frac{L_{\theta}}{\sqrt{n}}||\pmb{\mathbf{\theta}}(k)-\pmb{1}\bar{\theta}(k)^T||.
	\end{equation}
\end{lemma1}\label{yin5.2}

The above two lemmas, providing the related upper bounds of functions, are derived by  virtue of the Lipschitz smooth assumption.
They are essential for the subsequent  convergence analysis.
\subsection{Supporting Lemmas}\label{supprotlemma}
In this subsection, we present some results concerning \emph{expected optimization error} $\mathbb{E}[||\bar{x}(k)-x_*||^2]$ and \emph{expected consensus error} $\mathbb{E}[||\pmb{\mathbf{x}}(k)-\pmb{1}\bar{x}(k)^T||^2]$ for core computational problem, while the discussion of parameter learning problem can be found in \cite{pu2021sharp}. For ease of presentation, we denote
\begin{align}
	U_1(k) \triangleq &\mathbb{E}[||\bar{x}(k)-x_*||^2],
	V_1(k) \triangleq \mathbb{E}[||\pmb{\mathbf{x}}(k)-\pmb{1}\bar{x}(k)^T||^2],\label{U_1V_1}\\
	U_2(k) \triangleq &\mathbb{E}[||\bar{\theta}(k)-\theta_*||^2],
	V_2(k) \triangleq \mathbb{E}[||\pmb{\mathbf{\theta}}(k)-\pmb{1}\bar{\theta}(k)^T||^2].\label{U_2V_2}
\end{align}

Next we will bound $U_1(k+1)$ and $V_1(k+1)$ by error terms at iteration $k$.
The precise proof of Lemma \ref{55.5} is in \cref{appD}.
\begin{lemma1}\label{55.5}
	Let Assumption \ref{assump. func_pro}$\sim$\ref{assump. graph} hold, under \cref{alg:CDSA},
	
	(\pmb{A}) Supposing stepsize $\alpha_k\leq \frac{1}{L_x}$, we have
	\begin{align}
		U_1(k+1)&\leq(1-\alpha_k\mu_x)^2U_1(k)+\frac{\alpha_k^2L_x^2}{n}V_1(k)+\frac{\alpha_k^2L_{\theta}^2}{n}V_2(k)\notag\\
		&\quad+\frac{2L_xL_{\theta}\alpha_k^2}{n}\mathbb{E}[||\pmb{\mathbf{x}}(k)-\pmb{1}\bar{x}(k)^T||||\pmb{\theta}(k)-\pmb{1}\bar{\theta}(k)^T||]\notag\\
		&\quad+\frac{2\alpha_kL_x}{\sqrt{n}}(1-\alpha_k\mu_x)\mathbb{E}[||\bar{x}(k)-x_*||||\pmb{\mathbf{x}}(k)-\pmb{1}\bar{x}(k)^T||]\notag\\
		&\quad+\frac{2\alpha_kL_{\theta}}{\sqrt{n}}(1-\alpha_k\mu_x)\mathbb{E}[||\bar{x}(k)-x_*||||\pmb{\theta}(k)-\pmb{1}\bar{\theta}(k)^T||]\notag\\
		&\quad+\alpha_k^2\left(\frac{3M_xL_x^2}{n^2}\mathbb{E}[||\pmb{\mathbf{x}}(k)-\pmb{1}x_*^T||]+\frac{3M_xL_{\theta}^2}{n^2}\mathbb{E}[||\pmb{\theta}(k)-\pmb{1}\theta_*^T||]+\frac{\bar{M}}{n}\right)	\label{non_dim.recursion},
	\end{align}
	
	(\pmb{B}) Supposing stepsize $\alpha_k\leq \min\{\frac{1}{L_x},\frac{1}{3\mu_x}\}$, we have
	\begin{align}
		U_1(k+1)\leq&(1-\frac{3}{2}\alpha_k\mu_x)U_1(k)+\frac{6\alpha_kL_x^2}{n\mu_x}V_1(k)+\frac{6\alpha_kL_{\theta}^2}{n\mu_x}V_2(k)\notag\\
		&+\alpha_k^2\left(\frac{3M_xL_x^2}{n^2}\mathbb{E}[||\pmb{\mathbf{x}}(k)-\pmb{1}x_*^T||^2]+\frac{3M_xL_{\theta}^2}{n^2}\mathbb{E}[||\pmb{\theta}(k)-\pmb{1}\theta_*^T||^2]+\frac{\bar{M}}{n}\right)\label{u_recursion}.
	\end{align}
\end{lemma1}
%

Result \pmb{B} restricts the stepsize to smaller one than  that of Result \pmb{A}. 
It thus simplifies Result \pmb{A} \cref{non_dim.recursion} by removing the cross term to facilitate the later analysis.
We can revisit Inequality \cref{u_recursion} and reformulate it as $U_1(k+1)\leq (1-\frac{3}{2}\alpha_k\mu_x)U_1(k)+error(\alpha_k)$, where $error(\alpha)$ means an error function that is proportional to $\alpha$. We should mention that, since $\alpha_k>0$ and $\mu_x>0$, expected optimization error $U_1(k)$  roughly shrinks by a ratio $(1-\frac{3}{2}\alpha_k\mu_x)<1$. Though there is an error term related to $\alpha_k$,  when we choose diminishing stepsize policy and the consensus errors $V_1(k), V_2(k)$ as well as $\mathbb{E}(||\pmb{\mathbf{x}}(k)-\pmb{1}x_*^T||^2)$,  $\mathbb{E}(||\pmb{\theta}(k)-\pmb{1}\theta_*^T||^2)$ are bounded, the error may decrease to $0$, which indicates the convergence of $U_1(k)$.

We define
\begin{equation}
	\nabla_x \boldsymbol{F}( \pmb{\mathbf{x}},\pmb{\theta})\triangleq[\nabla_x f_1(x_1,\theta_1),\nabla_x f_2(x_2,\theta_2),\cdots,\nabla_x f_n(x_n,\theta_n)]^T\in \mathbb{R}^{n\times p}\label{f_vec}.
\end{equation} In the next lemma, we will show the recursive formulation of expected consensus error $V_1(k)$, which is  critical for convergence analysis. For completeness, we give its proof in \cref{appE}.
\begin{lemma1}\label{55.7}
	Let Assumption \ref{assump. func_pro}$\sim$\ref{assump. graph} hold,  and consider  \cref{alg:CDSA}. Then for any $k\geq0$, we have
	\begin{align}
		V_1(k+1)\leq&\frac{3+\rho_w^2}{4}V_1(k)+\alpha_k^2\rho_w^2n\sigma_x^2+3\alpha_k^2\rho_w^2\left(\frac{3}{1-\rho_w^2}+M_x\right)\left(L_x^2\mathbb{E}[||\pmb{\mathbf{x}}(k)-\pmb{1}x_*^T||^2]\right.\notag\\
		&\left.+L_{\theta}^2\mathbb{E}[||\pmb{\mathbf{\theta}}(k)-\pmb{1}\theta_*^T||^2]+||\nabla_x\boldsymbol{F}(\pmb{1}x_*^T,\pmb{1}\theta_*^T)||^2\right)\label{v_recursion}.
	\end{align}
\end{lemma1}

The recursion of expected consensus error can be reformulate as $V_1(k+1)\leq \frac{3+\rho_w^2}{4}V_1(k)+error(\alpha_k^2\rho_w^2)$. It is worth mentioning that  $V_1(k)$ can roughly shrink by $\frac{3+\rho_w^2}{4}<1$ since $\rho_w<1$.  Note that the extra error term in the consensus error is proportional to $\alpha_k^2$,
compared to $U_1(k)$ with an error term  proportional to $\alpha_k$.
We might obtain a qualitative conclusion that expected consensus error decrease faster than expected optimization error. We will present the precise proof in the next part that consensus error decrease to $0$ at an order $\mathcal{O}(\frac{1}{k^2})$ while optimization error at $\mathcal{O}(\frac{1}{k})$.

\newtheorem{remark}{Remark}
\begin{remark} Recalling recursion of $U_1(k)$ in (\ref{u_recursion}) and recursion of $V_1(k)$ in (\ref{v_recursion}), we could  notice that the expected consensus error is more related to the network connectivity  $\rho_w$, which is natural because ``consensus" is induced from the distributed algorithm, while ``optimization" mainly comes from original optimization method such as stochastic gradient descent.
\end{remark}
\vspace{-10pt}
\subsection{Uniform Bound}
\label{uni.bound_stepsize}
From now on, we consider stepsize policy as follows
\begin{equation}
	\alpha_k \triangleq \frac{\beta}{\mu_x(k+K)},
	\gamma_k \triangleq \frac{\beta}{\mu_\theta(k+K)}, \quad \forall k,\label{5.29}
\end{equation}
where the $\beta$ is a positive constant , and
\begin{equation}
	K \triangleq \Bigg{\lceil}\max\left\{\frac{3\beta(1+M_x)L_x^2}{\mu_x^2}, \frac{3\beta(1+M_{\theta})L_{\theta}^2}{\mu_{\theta}^2}\right\}\Bigg{\rceil}\label{5.30}
\end{equation}
with $\lceil\cdot\rceil$ denoting the ceiling function.

Next, We present a lemma that derives a uniform bound on the iterates $\{ \pmb{\theta}(k)\},\{ \pmb{x}(k)\}$ generated by \cref{alg:CDSA}. Such a result is helpful for bounding the expected optimization error and expected consensus error.
\begin{lemma1}\label{55.8}
	Let Assumption \ref{assump. func_pro}$\sim$\ref{assump. graph} hold. Consider  \cref{alg:CDSA}  with stepsize policy (\ref{5.29}). We then obtain   from  \cite[Lemma 8]{pu2021sharp} that for all $k\geq0$,
	\begin{equation}
		\mathbb{E}[\|\theta_i(k)-\theta_*\|^2] \leq \hat{\Theta}_i\triangleq\max\left\{\|\theta_i(0)-\theta_*\|^2,\frac{9\|\nabla h_i(\theta_*)\|}{\mu_{\theta}}^2+\frac{\sigma_{\theta}^2}{(1+M_{\theta})L_{\theta}^2}\right\}\label{theta_bound}.
	\end{equation}
	Based on \eqref{theta_bound}, we can obtain the following result with $\hat{\Theta}\triangleq \sum\nolimits_{i=1}^{n}\hat{\Theta}_i$,
	\begin{align}
		\mathbb{E}[||\pmb{\mathbf{x}}(k)-\pmb{1}x_*^T||^2]&\leq \hat{X}, {\rm~where ~} \label{bd-xk}
		\\ 
		\hat{X} \triangleq \max \Bigg{\{}||\pmb{\mathbf{x}}(0)-\pmb{1}x_*^T||^2,&\frac{11L_{\theta}^2 \hat{\Theta}}{\mu_x^2}
		+\frac{11||\nabla_x\boldsymbol{F}(\pmb{1}x_*^T, \pmb{1}\theta_*^T)||^2}{\mu_x^2}+\frac{7n\sigma_x^2}{9(1+M_x)L_x^2}\Bigg{\}}\label{5.31}.
	\end{align}
\end{lemma1}

We will give the proof of \eqref{bd-xk} in \cref{appC}.
Lemma \ref{55.8} indicates that although the problem we consider is unconstrained, the gap between the iterates generated by algorithm CDSA and the optimal solution is uniformly  bounded. It is critical for the analysis of sublinear convergence rates of $U_1(k)$ and $V_1(k)$. Then  based on this lemma, we  will provide uniform upper bounds for the expected optimization error and expected consensus error.

\begin{lemma1}\label{55.9}
	Let Assumption \ref{assump. func_pro}$\sim$\ref{assump. graph} hold.
	Consider   \cref{alg:CDSA}  with stepsize policy (\ref{5.29}). We then have
	$ U_1(k)\leq\frac{\hat{X}}{n},
	V_1(k)\leq4\hat{X}.$
\end{lemma1}
\newenvironment{proof}{{\noindent\it Proof}\quad}{\hfill $\square$}
\begin{proof} By recalling \eqref{bd-xk} and using Cauchy-Schiwaz inequality, we obtain that
	\begin{align*} U_1(k)&=\mathbb{E}[||\bar{x}(k)-x_*||^2]=
		\mathbb{E}\left[\left \|\frac{1}{n}\sum\nolimits_{i=1}^{n}x_i(k)-\frac{1}{n}\sum\nolimits_{i=1}^{n}x_*\right\|^2 \right]  \notag\\
		&=\frac{1}{n^2}\mathbb{E}\left[\left \|\sum\nolimits_{i=1}^{n}(x_i(k)-x_*)\right\|^2 \right]
		\leq\frac{1}{n^2}\times n\mathbb{E}[||\pmb{\mathbf{x}}(k)-\pmb{1}x_*^T||^2] \leq\frac{\hat{X}}{n},\notag\\
		V_1(k)&=\mathbb{E}[||\pmb{\mathbf{x}}(k)-\pmb{1}\bar{x}(k)^T||^2] =\mathbb{E}[||\pmb{\mathbf{x}}(k)-\pmb{1}x_*^T+\pmb{1}x_*^T-\pmb{1}\bar{x}(k)^T||^2]\notag\\
		&\leq2\mathbb{E}[||\pmb{\mathbf{x}}(k)-\pmb{1}x_*^T||^2]+2\mathbb{E}[||\pmb{1}(x_*-\bar{x}(k)||^2)]\\
		&\leq2\hat{X}+2n\times\frac{\hat{X}}{n}
		\leq 4\hat{X}.
	\end{align*}
\end{proof}
\vspace{-15pt}
\section{Main Results}
\label{sec:mainreslut}
In this section, we will make full use of previous results and then give a precise convergence rate analysis of  \cref{alg:CDSA}. The elaboration will be divided into three parts.
Firstly, we respectively establish the $\mathcal{O}(\frac{1}{k})$ and $\mathcal{O}(\frac{1}{k^2})$  convergence rate of $U_1(k)=\mathbb{E}[||\bar{x}(k)-x_*||^2]$
and $V_1(k)=\mathbb{E}[||\pmb{\mathbf{x}}(k)-\pmb{1}\bar{x}(k)^T||^2]$ based on two supporting lemmas in section \ref{supprotlemma}. Secondly, we show that the convergence rate, measured  by the mean-squared error of the decision variables,  is as follows.
\begin{align*}
	\frac{1}{n}\sum\nolimits_{i=1}^{n}\mathbb{E}[||x_i(k)-x_*||^2]\leq \frac{\beta^2\bar{M}}{(2\beta-1)n\mu_x^2(k+K)}+\frac{\mathcal{O}\left(\frac{1}{\sqrt{n}(1-\rho_w)}\right)}{(k+K)^{1.5}}+\frac{\mathcal{O}\left(\frac{1}{(1-\rho_w)^2}\right)}{(k+K)^2},
\end{align*}
where the first term is only concerned with the stochastic gradient descent  method which is network independent, while the higher-order depends on $(1-\rho_w)$. Finally, we characterize the transient time needed for CDSA to approach the asymptotic convergence rate is $\mathcal{O}\left(\frac{n}{(1-\rho_w)^2}\right)$. 
\subsection{Sublinear Convergence}\label{subconvergence}
We first prove that the expected consensus error  $V_1(k) = \mathbb{E}[||\pmb{\mathbf{x}}(k)-\pmb{1}\bar{x}(k)^T||^2] $ decays with rate $V_1(k)=\mathcal{O}(\frac{1}{k^2})$.

\begin{lemma1}\label{55.10}
	Let Assumption \ref{assump. func_pro}$\sim$\ref{assump. graph} hold.
	Consider \cref{alg:CDSA} with stepsize (\ref{5.29}). Recall the definitions of $K$.  Define
	\begin{align}
		\nabla \pmb{H} (\boldsymbol{\theta})&\triangleq \left[\nabla h_1(\theta_1), \nabla h_2(\theta_2),\cdots \nabla h_n(\theta_n)\right]\in \mathbb{R}^{n\times q}, \label{nable_H}\\
		K_1 &\triangleq \Big{\lceil}\max\Big{\{}2K,\frac{16}{1-\rho_w^2}\Big{\}}\Big{\rceil}\label{K_1}.
	\end{align}
	We then obtain from \cite[Lemma 10]{pu2021sharp} that  for any $k\geq K_1-K$, 
	\begin{align}
		V_2(k)&\leq\frac{\hat{V}_2}{(k+K)^2} {\rm~with~}	\hat{V}_2\triangleq\max\Big{\{}K_1^2\hat{\Theta},\frac{8\beta^2\rho_w^2c_4^{'}}{\mu_{\theta}^2(1-\rho_w^2)}\Big{\}}\label{hatV_2},\\
		{\rm~where~}	
		c_4^{'}\triangleq2&\left(\frac{3}{1-\rho_w^2}+M_{\theta}\right)\left(L_{\theta}^2\hat{\Theta}+\|\nabla \pmb{H}(\pmb{1}\theta_*^T)\|\right)+n\sigma_{\theta}^2\label{c_4p}.
	\end{align}
	Furthermore, we achieve that
	\begin{align}
		V_1(k)&\leq\frac{\hat{V}_1}{(k+K)^2}\label{V_order}
		{\rm~with~}
		\hat{V}_1 \triangleq \max \left\{4K_1^2\hat{X},\frac{8\beta^2\rho_w^2c_4}{\mu_x^2(1-\rho_w^2)}\right\},
		\\	{~\rm where~}	c_4 &\triangleq 3\left(\frac{3}{1-\rho_w^2}+M_x\right)\left(L_x^2\hat{X}+L_{\theta}^2\hat{\Theta}+||\nabla_x\boldsymbol{F}(\pmb{1}x_*^T,\pmb{1}\theta_*^T)||^2\right)+n\sigma_x^2\label{5.54}.
	\end{align}
\end{lemma1}
\begin{proof}
	We now prove \eqref{V_order}. From Lemma \ref{55.7} and \ref{55.8} it follows that
	\begin{equation}\label{recu-v1k}
		V_1(k+1)\leq\frac{3+\rho_w^2}{4}V_1(k)+\alpha_k^2\rho_w^2c_4,\quad \forall k\geq 0.
	\end{equation}
	We use induction method to  show that \eqref{V_order} holds   for any $k\geq K_1-K$.
	Recall from Lemma \ref{55.9}  that  $V_1(k) \leq 4\hat{X} $.
	Then for $k=K_1-K$, $ V_1(k) \leq \frac{ 4K_1^2\hat{X} }{K_1^2} =\frac{ 4K_1^2\hat{X} }{(k+K)^2} \leq   \frac{\hat{V}_1}{(k+K)^2}$ by the definition of $\hat{V}_1$ in (\ref{V_order}).
	Suppose that \eqref{V_order} holds for $k=\tilde{k}$. It suffices to show that \eqref{V_order} holds for $k=\tilde{k}+1$.
	
	
	Note from \eqref{K_1}  that  $\tilde{k}+K\geq \frac{16}{1-\rho_w^2}$ for any $\tilde{k}\geq K_1-K$.
	We then  have
	\begin{align*}
		\left(\frac{\tilde{k}+K}{\tilde{k}+K+1}\right)^2-\frac{3+\rho_w^2}{4}&
		=1-\frac{2}{\tilde{k}+K+1}+\frac{1}{(\tilde{k}+K+1)^2}-\frac{3+\rho_w^2}{4}\\
		&\geq\frac{1-\rho_w^2}{4}  -\frac{2}{\tilde{k}+K}  \geq \frac{1-\rho_w^2}{8}.
	\end{align*}
	Divide both sides of above inequality by $\frac{\beta^2\rho_w^2c_4}{\mu_x^2}$. Recalling the  definition of $\hat{V}_1$ in \eqref{V_order}, we have
	\begin{equation} 
			\frac{\beta^2\rho_w^2c_4}{\mu_x^2}\left(\left(\frac{\tilde{k}+K}{\tilde{k}+K+1}\right)^2-\frac{3+\rho_w^2}{4}\right)^{-1}\leq \frac{8\beta^2\rho_w^2c_4}{\mu_x^2(1-\rho_w^2)}\leq\hat{V}_1.
	\end{equation}
	This implies that
	\begin{equation}
		\frac{3+\rho_w^2}{4}\frac{\hat{V}_1}{(\tilde{k}+K)^2}+
		\frac{\beta^2\rho_w^2c_4}{\mu_x^2} \frac{1}{(\tilde{k}+K)^2} \leq\frac{\hat{V}_1}{(\tilde{k}+K+1)^2}. 
	\end{equation}
	Then by using \eqref{recu-v1k} and the definition of $\alpha_k$ in \eqref{5.29}, we derive that
	$V_1(\tilde{k}+1) \leq \frac{\hat{V}_1}{(\tilde{k}+K+1)^2},$ i.e.,
	\eqref{V_order} holds for $k=\tilde{k}+1$. Then   the  lemma  is proved.
\end{proof}


In light of Lemma \ref{55.10} and other auxiliary results, we establish the $\mathcal{O}(\frac{1}{k})$ convergence rate of expected optimization error $U_1(k)=\mathbb{E}[||\bar{x}(k)-x_*||^2]$ in the following lemma.

\begin{lemma1}\label{55.12}
	Let Assumption \ref{assump. func_pro}$\sim$\ref{assump. graph} hold.
	Consider \cref{alg:CDSA} with stepsize (\ref{5.29}), where $\beta>2$. We then have
	\begin{equation}
		\begin{aligned}
			U_1(k)\leq& \frac{\beta^2c_5}{(1.5\beta-1)n\mu_x^2(k+K)}+\frac{(K_1+K)^{1.5\beta}}{(k+K)^{1.5\beta}}\frac{\hat{X}}{n}
			\notag\\&
			+\left[\frac{3\beta^2(1.5\beta-1)c_5}{(1.5\beta-2)n\mu_x^2}+\frac{12\beta L_x^2\hat{V}_1}{(1.5\beta-2)n\mu_x^2}+\frac{12\beta L_{\theta}^2\hat{V}_2}{(1.5\beta-2)n\mu_x^2}\right]\cdot\frac{1}{(k+K)^2}
		\end{aligned}
	\end{equation}
	for any $k\geq K_1-K$, where
	\begin{equation}
		c_5 \triangleq \frac{3M_xL_x^2}{n}\hat{X}+\frac{3M_xL_{\theta}^2}{n}\hat{\Theta}+\bar{M}\label{5.63},
	\end{equation}
	$\hat{X}, K_1,\hat{V}_2, \hat{V}_1$, $\bar{M}$ are defined by (\ref{5.31}) (\ref{K_1}) (\ref{hatV_2}) (\ref{V_order}) (\ref{barM}) respectively.
\end{lemma1}
\begin{proof}
	Since $\alpha_k=\frac{\beta}{\mu_x(k+K)}$ by \eqref{5.29}, recalling the definition of $K$ and $K_1$ in \eqref{5.30} and \eqref{K_1}, we can see that $\alpha_k\leq\frac{\beta}{\mu K_1}\leq\frac{\beta}{2\mu K}\leq\frac{\mu_x}{6(1+M_x)L_x^2}\leq\min\{\frac{1}{3\mu_x},\frac{1}{L_x}\}$. Then  Lemma \ref{55.5}(B) holds.
	Together with \ref{55.8} it follows that for any $k\geq K_1-K$,
	\begin{equation} 
		U_1(k+1)\leq(1-\frac{3}{2}\alpha_k\mu_x)U_1(k)
		+\frac{6\alpha_kL_x^2}{n\mu_x}V_1(k)+\frac{6\alpha_kL_{\theta}^2}{n\mu_x}V_2(k)+\frac{\alpha_k^2c_5}{n}.
	\end{equation}
	Recalling the definition of $\alpha_k=\frac{\beta}{\mu_x(k+K)}$, we have
	\begin{equation}
		U_1(k+1)\leq (1-\frac{3\beta}{2(k+K)})U_1(k)+\frac{6\beta L_x^2V_1(k)}{n\mu_x^2(k+K)}+\frac{6\beta L_{\theta}^2V_2(k)}{n\mu_x^2(k+K)}+\frac{\beta^2c_5}{n\mu_x^2}\cdot\frac{1}{(k+K)^2}.
	\end{equation}
	Thus
	\begin{equation}
		\begin{aligned}
			&U_1(k)\leq \prod\nolimits_{t=K_1+K}^{k+K-1}(1-\frac{3\beta}{2t})U_1(K_1)\notag\\
			&+\sum\nolimits_{t=K_1+K}^{k+K-1} \prod\nolimits_{j=t+1}^{k+K-1}(1-\frac{3\beta}{2j}) \left(\frac{6\beta L_x^2}{n\mu_x^2}\cdot\frac{V_1(t-K)}{t}+\frac{6\beta L_{\theta}^2}{n\mu_x^2}\cdot\frac{V_2(t-K)}{t}+\frac{\beta^2c_5}{n\mu_x^2}\cdot\frac{1}{t^2}\right).
		\end{aligned}
	\end{equation}
	Recall from \cite[lemma 11]{pu2021sharp} that
	for any $\forall 1<j<k, j\in\mathbb{N}$ and $1<\gamma\leq j/2$, $\prod\nolimits_{t=j}^{k-1}(1-\frac{\gamma}{t})\leq\frac{j^{\gamma}}{k^{\gamma}}$. Then we achieve
	\begin{equation}
		\begin{aligned}
			&U_1(k)\leq  \frac{(K_1+K)^{1.5\beta}}{(k+K)^{1.5\beta}}U_1(K_1)\notag\\
			&+\sum\nolimits_{t=K_1+K}^{k+K-1}\frac{(t+1)^{1.5\beta}}{(k+K)^{1.5\beta}}\left(\frac{6\beta L_x^2}{n\mu_x^2}\cdot\frac{V_1(t-K)}{t}+\frac{6\beta L_{\theta}^2}{n\mu_x^2}\cdot\frac{V_2(t-K)}{t}+\frac{\beta^2c_5}{n\mu_x^2}\cdot\frac{1}{t^2}\right)\notag\\
			&=\frac{(K_1+K)^{1.5\beta}}{(k+K)^{1.5\beta}}U_1(K_1)+\frac{6\beta L_{\theta}^2}{n\mu_x^2(k+K)^{1.5\beta}}\sum\nolimits_{t=K_1+K}^{k+K-1}\frac{(t+1)^{1.5\beta}V_2(t-K)}{t}+\notag\\
			&\frac{6\beta L_x^2}{n\mu_x^2(k+K)^{1.5\beta}}\sum\nolimits_{t=K_1+K}^{k+K-1}\frac{(t+1)^{1.5\beta}V_1(t-K)}{t}+\frac{\beta^2c_5}{n\mu_x^2(k+K)^{1.5\beta}}\sum\nolimits_{t=K_1+K}^{k+K-1}\frac{(t+1)^{1.5\beta}}{t^2}.\notag
		\end{aligned}
	\end{equation}
	In light of Lemma \ref{55.10}, we have $V_1(k-K)\leq\frac{\hat{V}_1}{k^2}$ and $V_2(k-K)\leq\frac{\hat{V}_2}{k^2}$  for any  $k\geq K_1-K$. Hence
	\begin{align}
		U_1(k)\leq &\frac{\beta^2c_5}{n\mu_x^2(k+K)^{1.5\beta}}\sum\nolimits_{t=K_1+K}^{k+K-1}\frac{(t+1)^{1.5\beta}}{t^2}
		+\frac{(K_1+K)^{1.5\beta}}{(k+K)^{1.5\beta}}U_1(K_1)\notag\\
		&+\frac{6\beta L_x^2\hat{V}_1}{n\mu_x^2(k+K)^{1.5\beta}}\sum\nolimits_{t=K_1+K}^{k+K-1}\frac{(t+1)^{1.5\beta}}{t^3}+\frac{6\beta L_{\theta}^2\hat{V}_2}{n\mu_x^2(k+K)^{1.5\beta}}\sum\nolimits_{t=K_1+K}^{k+K-1}\frac{(t+1)^{1.5\beta}}{t^3}.\label{tuoterUtemp}
	\end{align}
	By the proof in \cite[lemma 12]{pu2021sharp}, when $b>a\geq K_1$, we have
	\begin{equation}
		\sum\nolimits_{t=a}^{b}\frac{(t+1)^{1.5\beta}}{t^2}\leq \frac{b^{1.5\beta-1}}{1.5\beta-1}+\frac{3(1.5\beta-1)b^{1.5\beta-2}}{1.5\beta-2} ,~
		\sum\nolimits_{t=a}^{b}\frac{(t+1)^{1.5\beta}}{t^3}\leq \frac{2b^{1.5\beta-2}}{1.5\beta-2}.\label{t_squre_cube}
	\end{equation}
	Thus
	\begin{equation}
		\begin{aligned}
			U_1(k)\leq&\frac{\beta^2 c_5}{(1.5\beta-1)n\mu_x^2(k+K)}+\frac{3\beta^2(1.5\beta-1)c_5}{(1.5\beta-2)n\mu_x^2}\cdot\frac{1}{(k+K)^2}+\frac{(K_1+K)^{1.5\beta}}{(k+K)^{1.5\beta}}U_1(K_1)\\
			&+\frac{12\beta L_x^2 \hat{V}_1}{(1.5\beta-2)n\mu_x^2}\cdot\frac{1}{(k+K)^2}+\frac{12\beta L_{\theta}^2 \hat{V}_2}{(1.5\beta-2)n\mu_x^2}\cdot\frac{1}{(k+K)^2}.
		\end{aligned}
	\end{equation}
	Recalling Lemma \ref{55.9} yields the desired result.
\end{proof}


\subsection{Rate Estimate}\label{rate}
In this subsection, we will discuss the factors that affect the convergence rate of the algorithm, especially the network size $n$, the spectral gap $(1-\rho_w)$, the summation of initial optimization errors $ \sum\nolimits_{i=1}^{n}||x_i(0)-x_*||^2$ and consensus errors $\sum\nolimits_{i=1}^{n}||\theta_i(0)-\theta_*||^2$, and the heterogenous of  computational functions and learning functions characterized by $\sum\nolimits_{i=1}^{n}||\nabla_xf_i(x_*,\theta_*)||^2$ and $\sum\nolimits_{i=1}^{n}||\nabla h_i(\theta_*)||^2$.
Firstly, we bound the constants appearing in Lemmas \ref{55.10} and \ref{55.12} by the  aforementioned  factors.  We then  utilize them to simplify the sublinear  rate of the expected optimization error, based on which,  we can improve the convergence rate and derive the main result for Algorithm  \ref{alg:CDSA}.


\begin{lemma1}\label{rate1}
	Denote $A_1 \triangleq \sum\nolimits_{i=1}^{n}||x_i(0)-x_*||^2, B_1 \triangleq \sum\nolimits_{i=1}^{n}||\nabla_xf_i(x_*;\theta_*)||^2,\\ A_2 \triangleq \sum\nolimits_{i=1}^{n}||\theta_i(0)-\theta_*||^2,$ and $ B_2 \triangleq \sum\nolimits_{i=1}^{n}||\nabla h_i(\theta_*)||^2$. Then the orders of  constants' $\hat{X}$ (\ref{5.31}), $\hat{\Theta}$ (\ref{theta_bound}), $\hat{V}_1$ (\ref{V_order}), $\hat{V}_2$ (\ref{hatV_2}), $c_4$ (\ref{5.54}) and $c_5$ (\ref{5.63}) are    as follow.
	\begin{align*}
		\hat{X}=\mathcal{O}(A_1+A_2+B_1+B_2+n),&\quad \hat{\Theta}=\mathcal{O}(A_2+B_2+n),\\
		\hat{V}_1=\mathcal{O}\left(\frac{A_1+A_2+B_1+B_2+n}{(1-\rho_w)^2}\right),&\quad\hat{V}_2=\mathcal{O}\left(\frac{A_2+B_2+n}{(1-\rho_w)^2}\right),\\
		c_4=\mathcal{O}\left(\frac{A_1+A_2+B_1+B_2+n}{1-\rho_w}\right),&\quad c_5=\mathcal{O}\left(\frac{A_1+A_2+B_1+B_2+n}{n}\right).
	\end{align*}
\end{lemma1}
\begin{proof}
	The upper bound of $\hat{\Theta}$ and $\hat{V}_2$ deal only with unknown parameter $\theta$, which can be inherited directly from \cite[lemma 13]{pu2021sharp}. As for $\hat{X}$, recalling (\ref{5.31}) we have
	\begin{equation}
		\begin{aligned} \hat{X}\leq&||\pmb{\mathbf{x}}(0)-\pmb{1}x_*^T||^2+\frac{11L_{\theta}^2\hat{\Theta}}{\mu_x^2}
			+\frac{11||\nabla_xF(\pmb{1}x_*^T,\pmb{1}\theta_*^T)||^2}{\mu_x^2}+
			\frac{7n\sigma_x^2}{9(1+M_x)L_x^2}\\
			=&\mathcal{O}(A_1+A_2+B_1+B_2+n).
		\end{aligned}
	\end{equation}
	From the definition of $c_4$ in (\ref{5.54}), it follows that
	\begin{equation}
		c_4=3\left(\frac{3}{1-\rho_w^2}+M_x\right)(L_x^2\hat{X}+L_{\theta}^2\hat{\Theta}+||\nabla_xF(\pmb{1}x_*^T,\pmb{1}\theta_*^T)||^2)+n\sigma_x^2=\mathcal{O}\left(\frac{A_1+A_2+B_1+B_2+n}{1-\rho_w}\right)
	\end{equation}
	Note from  \eqref{K_1} and \eqref{5.30} that $K_1=\mathcal{O}\left(\frac{1}{1-\rho_w}\right)$. Then by  \eqref{V_order} , we obtain
	\begin{equation}
		\hat{V}_1=\max \left \{4K_1^2\hat{X},\frac{8\beta^2\rho_w^2c_4}{\mu_x^2(1-\rho_w^2)}\right \}=\mathcal{O}\left(\frac{A_1+A_2+B_1+B_2+n}{(1-\rho_w)^2}\right).
	\end{equation}
	In light of equation (\ref{5.63}), we can achieve
	\begin{equation}
		c_5=\frac{3M_xL_x^2}{n}\hat{X}+\frac{3M_xL_{\theta}^2}{n}\hat{\Theta}+\bar{M_x}=\mathcal{O}\left(\frac{A_1+A_2+B_1+B_2+n}{n} \right).
	\end{equation}
\end{proof}

The simplification of these constants makes it convenient for later analysis. In light of relation (\ref{V_order}), since $\hat{V}_1$ is the only constant, the convergence result of expected consensus error $V_1(k)$ can be easily obtained. While the expected optimization error $U_1(k)$ needs to be reformulated more concisely.
\newtheorem{corollary1}{Corollary}[section]
\begin{corollary1}\label{coro1}
	Let Assumption \ref{assump. func_pro}$\sim$\ref{assump. graph} hold. Consider \cref{alg:CDSA} with stepsize policy (\ref{5.29}), where $\beta>2$.  Then   we
	obtain from \cite[Corollary 1]{pu2021sharp} that
	\begin{align*}
		U_2(k)\leq \frac{\beta^2c_5'}{(1.5\beta-1)n\mu_x^2}\cdot\frac{1}{(k+K)}+\frac{c_6'}{(k+K)^2}, \quad \forall k\geq K_1-K,
	\end{align*}
	where $c_5'\triangleq \frac{2M_{\theta}L_{\theta}^2}{n}\hat{\Theta}+\frac{2M_{\theta}\sum\nolimits_{i=1}^{n}\|\nabla h_i(\theta_*)\|^2}{n}+\sigma_{\theta}^2$, $c_6'=O\left(\frac{A_2+B_2+n}{n(1-\rho_w)^2}\right).$
	Based on which, we further have that for any $k\geq K_1-K,$
	\begin{equation}
		U_1(k)\leq \frac{\beta^2c_5}{(1.5\beta-1)n\mu_x^2}\cdot\frac{1}{(k+K)}+\frac{c_6}{(k+K)^2},
	\end{equation}
	where $c_5$ is defined in (\ref{5.63}), and $c_6=O\left(\frac{A_1+A_2+B_1+B_2+n}{n(1-\rho_w)^2}\right)$.
\end{corollary1}
\begin{proof}
	In light of Lemma \ref{55.12} and Lemma \ref{rate1}, we can obtain that
	\begin{align*}
		U_1(k)&\leq\frac{\beta^2c_5}{(1.5\beta-1)n\mu_x^2(k+K)}+\frac{(K_1+K)^{1.5\beta-2}}{(k+K)^{1.5\beta-2}}\frac{\hat{X}}{n}\cdot\frac{1}{(k+K)^2}\notag\\
		&\qquad+\left[\frac{3\beta^2(1.5\beta-1)c_5}{(1.5\beta-2)n\mu_x^2}+\frac{12\beta L_x^2\hat{V}_1}{(1.5\beta-1)n\mu_x^2}+\frac{12\beta L_{\theta}^2\hat{V}_2}{(1.5\beta-1)n\mu_x^2}\right]\cdot\frac{1}{(k+K)^2}\\
		&=\frac{\beta^2c_5}{(1.5\beta-1)n\mu_x^2}\cdot\frac{1}{(k+K)}+\mathcal{O}\left(\frac{A_1+A_2+B_1+B_2+n}{n}\right)\frac{1}{(k+K)^2}\\
		&\quad+\left[\mathcal{O}\left(\frac{A_1+A_2+B_1+B_2+n}{n^2} \right)+\mathcal{O}\left(\frac{A_1+A_2+B_1+B_2+n}{n(1-\rho_w)^2}\right)+\mathcal{O}\left(\frac{A_2+B_2+n}{n(1-\rho_w)^2}\right)\right]\frac{1}{(k+K)^2}\\
		&\leq\frac{\beta^2c_5}{(1.5\beta-1)n\mu_x^2}\cdot\frac{1}{(k+K)}+\mathcal{O}\left(\frac{A_1+A_2+B_1+B_2+n}{n(1-\rho_w)^2}\right)\frac{1}{(k+K)^2}.
	\end{align*}
\end{proof}

Based on this corollary, together with Lemma \ref{55.5}, we further elaborate the convergence result of Algorithm \ref{alg:CDSA}. 
Especially, we give an upper bound of $\frac{1}{n} \sum\nolimits_{i=1}^n \mathbb{E}\left[\left\|x_i(k)-x_*\right\|^2\right]$ and formulate it in a way to make  an intuitive comparison with the centralized algorithm.
\newtheorem{theory}{Theorem}[section]
\begin{theory}
	Let Assumption \ref{assump. func_pro}$\sim$\ref{assump. graph} hold. Consider  \cref{alg:CDSA} with stepsize policy (\ref{5.29}), where $\beta>2$. Then for any  $k\geq K_1-K$, we have
	\begin{align}
		&	\frac{1}{n}\sum\nolimits_{i=1}^{n}\mathbb{E}[||x_i(k)-x_*||^2]\leq \frac{\beta^2\bar{M}}{(2\beta-1)n\mu_x^2(k+K)}\notag\\		
		&\quad+\mathcal{O}\left(\frac{A_1+A_2+B_1+B_2+n}{n\sqrt{n}(1-\rho_w)}\right)\frac{1}{(k+K)^{1.5}}+\mathcal{O}\left(\frac{A_1+A_2+B_1+B_2+n}{n(1-\rho_w)^2} \right)\frac{1}{(k+K)^2},\label{mainreslut}
	\end{align}
	where $\bar{M}$ is defined in (\ref{barM}).
\end{theory}
\begin{proof}
	For $k\geq K_1-K$, by recalling Lemma \ref{55.5}(\pmb{A}) and the definition of $U_1(k),V_1(k)$ and $U_2(k), V_2(k)$ in (\ref{U_1V_1}) and (\ref{U_2V_2}), we have
	\begin{align*}
		&U_1(k+1)
		\leq(1-\alpha_k\mu_x)^2U_1(k)+\frac{\alpha_k^2L_x^2}{n}V_1(k)+\frac{\alpha_k^2L_{\theta}^2}{n}V_2(k)+\frac{2L_xL_{\theta}\alpha_k^2}{n}\sqrt{V_1(k)V_2(k)}\notag\\
		&\qquad+\frac{2\alpha_kL_x}{\sqrt{n}}\sqrt{U_1(k)V_1(k)}+\frac{2\alpha_kL_{\theta}}{\sqrt{n}}\sqrt{U_1(k)V_2(k)}\notag\\
		&\qquad+\alpha_k^2\left(\frac{3M_xL_x^2}{n^2}(nU_1(k)+V_1(k))+\frac{3M_xL_{\theta}^2}{n^2}(nU_2(k)+V_2(k))+\frac{\bar{M}}{n}\right)\\
		&=(1-2\alpha_k\mu_x)U_1(k)+\alpha_k^2\left(\mu_x^2+\frac{3M_xL_x^2}{n}\right)U_1(k)+\alpha_k^2\cdot\frac{3M_xL_{\theta}^2}{n}U_2(k)\\
		&\qquad+\frac{\alpha_k^2L_x^2}{n}\left(1+\frac{3M_x}{n}\right)V_1(k)+\frac{\alpha_k^2L_{\theta}^2}{n}\left (1+\frac{3M_x}{n}\right)V_2(k) +\frac{2L_xL_{\theta}\alpha_k^2}{n}\sqrt{V_1(k)V_2(k)}\\
		&\qquad+\frac{2\alpha_kL_x}{\sqrt{n}}\sqrt{U_1(k)V_1(k)}+\frac{2\alpha_kL_{\theta}}{\sqrt{n}}\sqrt{U_1(k)V_2(k)}+\frac{\alpha_k^2 \bar{M}}{n},
	\end{align*}
	where the first inequality follows the Cauchy-Schwarz inequality in the probabilistic form and the fact that
	\begin{align}\label{bd-mse}
		\mathbb{E}[||\pmb{\mathbf{x}}(k)-\pmb{1}x_*^T||^2]&=\mathbb{E}\left[\|\mathbf{x}(k)-\pmb{1}\bar{x}^T+\pmb{1}\bar{x}^T-\pmb{1}x_*^T\|^2\right] \notag
		\\& \leq 2\mathbb{E}\left[\|\mathbf{x}(k)-\pmb{1}\bar{x}^T\|^2\right]+2\mathbb{E}\left[\|\pmb{1}\bar{x}^T-\pmb{1}x_*^T\|^2\right] \\
		&=2\mathbb{E}\left[\|\mathbf{x}(k)-\pmb{1}\bar{x}^T\|^2\right]+2n\mathbb{E}\left[\| \bar{x} - x_* \|^2\right] =2V_1(k)+2nU_1(k)\notag.
	\end{align}
	Thus, due to $\alpha_k=\frac{\beta}{\mu_x(k+K)}$, we have
	\begin{align*}
		U_1&(k+1)	\leq\left(1-\frac{2\beta}{k+K}\right)U_1(k)+\frac{\beta^2U_1(k)}{(k+K)^2}\left(1+\frac{3M_xL_x^2}{n\mu_x^2}\right)+\frac{3M_xL_{\theta}^2\beta^2U_2(k)}{n\mu_x^2(k+K)^2}\\
		&\qquad+\frac{\beta^2L_x^2}{n\mu_x^2}\left(1+\frac{3M_x}{n}\right)\frac{V_1(k)}{(k+K)^2}+\frac{\beta^2L_{\theta}^2}{n\mu_x^2}\left(1+\frac{3M_x}{n}\right)\frac{V_2(k)}{(k+K)^2}\\
		&+\frac{2L_xL_{\theta}\beta^2}{n\mu_x^2}\frac{\sqrt{V_1(k)V_2(k)}}{(k+K)^2}+\frac{2\beta L_x}{\sqrt{n}\mu_x}\frac{\sqrt{U_1(k)V_1(k)}}{k+K}+\frac{2\beta L_{\theta}}{\sqrt{n}\mu_x}\frac{\sqrt{U_1(k)V_2(k)}}{k+K}+\frac{\beta^2\bar{M}}{n\mu_x^2}\frac{1}{(k+K)^2}
	\end{align*}
	Denote by $c_7=1+\frac{3M_xL_x^2}{n\mu_x^2}$ and  $c_8=1+\frac{3M_x}{n}$ . Then in light of \cite[Lemma 11]{pu2021sharp}, we obtain that
	\begin{equation}
		\begin{aligned}
			&U_1(k)\leq\prod\nolimits_{t=K_1+K}^{k+K-1}\left(1-\frac{2\beta}{t}\right)U_1(K_1)+\sum\nolimits_{t=K_1+K}^{k+K-1}\left(\prod\nolimits_{i=t+1}^{k+K-1}\left(1-\frac{2\beta}{i}\right)\right)\Bigg[\frac{\beta^2\bar{M}}{n\mu_x^2t^2}\\
			&\quad+\frac{3M_xL_{\theta}^2\beta^2}{n\mu_x^2}\frac{U_2(t-K)}{t^2}+\frac{\beta^2L_x^2c_8}{n\mu_x^2}\frac{V_1(t-K)}{t^2}+\frac{\beta^2L_{\theta}^2c_8}{n\mu_x^2}\frac{V_2(t-K)}{t^2}+\frac{2L_xL_{\theta}\beta^2}{n\mu_x^2}\frac{\sqrt{V_1(t-K)V_2(t-K)}}{t^2}\\
			&\quad+\frac{\beta^2c_7U_1(t-K)}{t^2}+\frac{2\beta L_x}{\sqrt{n}\mu_x}\frac{\sqrt{U_1(t-K)V_1(t-K)}}{t}+\frac{2\beta L_{\theta}}{\sqrt{n}\mu_x}\frac{\sqrt{U_1(t-K)V_2(t-K)}}{t}\Bigg]\\
			&\leq
			\frac{(K_1+K)^{2\beta}}{(k+K)^{2\beta}}U_1(K_1)+\sum\nolimits_{t=K_1+K}^{k+K-1}\frac{(t+1)^{2\beta}}{(k+K)^{2\beta}}\Bigg[\frac{\beta^2\bar{M}}{n\mu_x^2t^2}+\frac{\beta^2c_7U_1(t-K)}{t^2}\\
			&\qquad+\frac{\beta^2L_x^2c_8}{n\mu_x^2}\frac{V_1(t-K)}{t^2}+\frac{\beta^2L_{\theta}^2c_8}{n\mu_x^2}\frac{V_2(t-K)}{t^2}+\frac{2L_xL_{\theta}\beta^2}{n\mu_x^2}\frac{\sqrt{V_1(t-K)V_2(t-K)}}{t^2}\\
			&\qquad\qquad+\frac{3M_xL_{\theta}^2\beta^2U_2(t-K)}{n\mu_x^2t^2}+\frac{2\beta L_x\sqrt{U_1(t-K)V_1(t-K)}}{\sqrt{n}\mu_xt}+\frac{2\beta L_{\theta}\sqrt{U_1(t-K)V_2(t-K)}}{\sqrt{n}\mu_xt}\Bigg].
		\end{aligned}
	\end{equation}
	
	According to Corollary \ref{coro1} and Lemma \ref{55.10}, we have
	\begin{align*}
		&U_1(k)\leq\frac{(K_1+K)^{2\beta}}{(k+K)^{2\beta}}U_1(K_1)+ \frac{1}{(k+K)^{2\beta}}\cdot\frac{\beta^2\bar{M}}{n\mu_x^2}\sum\nolimits_{t=K_1+K}^{k+K-1}\frac{(t+1)^{2\beta}}{t^2}\\
		&\quad+ \frac{\beta^2c_7}{(k+K)^{2\beta}}\sum\nolimits_{t=K_1+K}^{k+K-1}\frac{(t+1)^{2\beta}}{t^2}\left[\frac{\beta^2c_5}{(1.5\beta-1)n\mu_x^2}\cdot\frac{1}{t}+\frac{c_6}{t^2}\right]\\
		&\quad+\frac{3M_xL_{\theta}^2\beta^2}{n\mu_x^2(k+K)^{2\beta}}\sum\nolimits_{t=K_1+K}^{k+K-1}\frac{(t+1)^{2\beta}}{t^2}\left[\frac{\beta^2c_5^{'}}{(1.5\beta-1)n\mu_{\theta}^2}\cdot\frac{1}{t}+\frac{c_6^{'}}{t^2}\right]\\
		&\quad+\frac{\beta^2L_x^2c_8}{n\mu_x^2(k+K)^{2\beta}}\sum\nolimits_{t=K_1+K}^{k+K-1}\frac{(t+1)^{2\beta}}{t^2}\cdot\frac{\hat{V}_1}{t^2}\\
		&\quad+\frac{\beta^2L_{\theta}^2c_8}{n\mu_x^2(k+K)^{2\beta}}\sum\nolimits_{t=K_1+K}^{k+K-1}\frac{(t+1)^{2\beta}}{t^2}\cdot\frac{\hat{V}_2}{t^2}+\frac{2L_xL_{\theta}\beta^2}{n\mu_x^2(k+K)^{2\beta}}\sum\nolimits_{t=K_1+K}^{k+K-1}\frac{(t+1)^{2\beta}}{t^2}\frac{\sqrt{\hat{V}_1\hat{V}_2}}{t^2}\\
		&\quad+\frac{2\beta L_x}{\sqrt{n}\mu_x(k+K)^{2\beta}}\sum\nolimits_{t=K_1+K}^{k+K-1}\frac{(t+1)^{2\beta}}{t}\sqrt{\frac{\beta^2c_5}{(1.5\beta-1)n\mu_x^2}\cdot\frac{1}{t}+\frac{c_6}{t^2}}\sqrt{\frac{\hat{V}_1}{t^2}}\\
		&\quad+\frac{2\beta L_{\theta}}{\sqrt{n}\mu_x(k+K)^{2\beta}}\sum\nolimits_{t=K_1+K}^{k+K-1}\frac{(t+1)^{2\beta}}{t}\sqrt{\frac{\beta^2c_5}{(1.5\beta-1)n\mu_x^2}\cdot\frac{1}{t}+\frac{c_6}{t^2}}\sqrt{\frac{\hat{V}_2}{t^2}}.
	\end{align*}
	Since $\sqrt{a+b}\leq\sqrt{a}+\sqrt{b}$, we can achieve
	\begin{equation}
		\sqrt{\frac{\beta^2c_5}{(1.5\beta-1)n\mu_x^2}\cdot\frac{1}{t}+\frac{c_6}{t^2}}\cdot\sqrt{\frac{\hat{V}_1}{t^2}}\leq\beta \sqrt{\frac{c_5\hat{V}_1}{(1.5\beta-1)n\mu_x^2}}\cdot\frac{1}{t^{1.5}}+\frac{\sqrt{c_6\hat{V}_1}}{t^2},
	\end{equation}
	then
		\begin{align*}
			&	U_1(k)\leq\frac{\beta^2\bar{M}\sum_{t=K_1+K}^{k+K-1}\frac{(t+1)^{2\beta}}{t^2}}{(k+K)^{2\beta}n\mu_x^2}+\frac{(K_1+K)^{2\beta}U_1(K_1)}{(k+K)^{2\beta}}+\frac{2\beta^2\sqrt{c_5}(L_x\sqrt{\hat{V}_1}+L_{\theta}\sqrt{\hat{V}_2})\sum_{t=K_1+K}^{k+K-1}\frac{(t+1)^{2\beta}}{t^{2.5}}}{\sqrt{1.5\beta-1}\times n\mu_x^2(k+K)^{2\beta}}\\
			&+\frac{1}{(k+K)^{2\beta}}\left[\frac{\beta^4c_5c_7}{(1.5\beta-1)n\mu_x^2}+\frac{3M_x\beta^4L_{\theta}^2c_5^{'}}{(1.5\beta-1)n^2\mu_x^2\mu_{\theta}^2}+\frac{2\beta(L_x\sqrt{c_6\hat{V}_1}+L_{\theta}\sqrt{c_6^{'}\hat{V}_2})}{\sqrt{n}\mu_x}\right]\sum\nolimits_{t=K_1+K}^{k+K-1}\frac{(t+1)^{2\beta}}{t^{3}}\\
			&+\frac{1}{(k+K)^{2\beta}}\left[\beta^2c_6c_7+\frac{3M_x\beta^2L_{\theta}^2c_6^{'}}{n\mu_x^2}+\frac{\beta^2(L_x^2\hat{V}_1+L_{\theta}^2\hat{V}_2)c_8}{n\mu_x^2} +\frac{2L_{x}L_{\theta}\beta^2\sqrt{\hat{V}_1\hat{V}_2}}{n\mu_x^2}\right]\sum\nolimits_{t=K_1+K}^{k+K-1}\frac{(t+1)^{2\beta}}{t^{4}}.
		\end{align*}
	Recall \eqref{t_squre_cube}  and note that
	\begin{equation}
		\begin{aligned}
			\sum\nolimits_{t=a}^{b}\frac{(t+1)^{2\beta}}{t^{2.5}}&\leq\sum\nolimits_{t=a}^{b}\frac{2(t+1)^{2\beta}}{(t+1)^{2.5}}\leq\int_{a+1}^{b+1}2t^{2\beta-2.5}dt\leq\frac{2(b+1)^{2\beta-1.5}}{2\beta-1.5},\\
			\sum\nolimits_{t=a}^{b}\frac{(t+1)^{2\beta}}{t^{4}}&\leq\sum\nolimits_{t=a}^{b}\frac{2(t+1)^{2\beta}}{(t+1)^4}\leq\int_{a+1}^{b+1}2t^{2\beta-4}dt\leq\frac{2(b+1)^{2\beta-3}}{2\beta-4},\quad \forall a\geq 16.
		\end{aligned}
	\end{equation}
	Then by noticing that $c_7=c_8=\mathcal{O}(1)$ and   using Lemma \ref{rate1}, we have
		\begin{align*}
			&	U_1(k)\leq \frac{\beta^2\bar{M}}{(2\beta-1)n\mu_x^2(k+K)}+\mathcal{O}\left(\frac{A_1+A_2+B_1+B_2+n}{n\sqrt{n}(1-\rho_w)}\right)\frac{1}{(k+K)^{1.5}}+\mathcal{O}\left(\frac{A_1+A_2+B_1+B_2+n}{n(1-\rho_w)^2}\right)\frac{1}{(k+K)^2}\notag\\
			&~+\mathcal{O}\left(\frac{A_1+A_2+B_1+B_2+n}{n(1-\rho_w)^{2}}\right)\frac{1}{(k+K)^3}+\mathcal{O}\left(\frac{A_1+A_2+B_1+B_2+n}{n(1-\rho_w)^{2\beta}}\right)\frac{1}{(k+K)^{2\beta}}\\
			&	=\frac{\beta^2\bar{M}}{(2\beta-1)n\mu_x^2(k+K)}+\mathcal{O}\left(\frac{A_1+A_2+B_1+B_2+n}{n\sqrt{n}(1-\rho_w)}\right)\frac{1}{(k+K)^{1.5}}+\mathcal{O}\left(\frac{A_1+A_2+B_1+B_2+n}{n(1-\rho_w)^2}\right)\frac{1}{(k+K)^2}.
		\end{align*}
	By recalling \eqref{bd-mse}, we have $\frac{1}{n}\sum\nolimits_{i=1}^{n}\mathbb{E}[||x_i(k)-x_*||^2] \leq 2U_1(k)+\frac{2V_1(k)}{n}.$
This together with  \eqref{V_order} and the estimate of $\hat{V}_1$ in Lemma \ref{rate1} prove the   result.
\end{proof}

In light of relation (\ref{mainreslut}), by recalling the definitions of $A_1,A_2, B_1,B_2$ in Lemma \ref{rate1}, we can see that the convergence rate is proportional to  initial errors for both   computational problem $\sum\nolimits_{i=1}^n\left\|x_i(0)-x_*\right\|^2$ and parameter learning problem $\sum\nolimits_{i=1}^n\left\|\theta_i(0)-\theta_*\right\|^2$.
It is worth noting that the heterogeneity  of  agents' individual cost functions, measured by $ B_1= \sum\nolimits_{i=1}^{n}||\nabla_xf_i(x_*;\theta_*)||^2,~ B_2 = \sum\nolimits_{i=1}^{n}||\nabla h_i(\theta_*)||^2$, also influence the convergence rate in a similar   way.
Though $ \theta_* , x_*$ are  respectively the optimal solutions to $\min_{\theta} \frac{1}{n} \sum\nolimits_{i=1}^n h_i(\theta)$ and  $\min_{x} \frac{1}{n} \sum\nolimits_{i=1}^n f_i\left(x ; \theta^*\right)$,
they are usually not the optimal solution to each local function $ h_i(\theta),f_i(x,\theta).$ Therefore,  the bigger the difference between the local costs, the slower the convergence rate of the algorithm.

\begin{remark}  Here we give some comments regerading  the influence of the network size $n$ and the spectral gap $(1-\rho_w)$  on  the convergence rate.
Since $A_1, A_2,B_1$ and $B_2$ are all $\mathcal{O}(n)$, we can simplify the relation (\ref{mainreslut}) as follow.
\begin{equation}
	\frac{1}{n}\sum\nolimits_{i=1}^{n}\mathbb{E}[||x_i(k)-x_*||^2]\leq \frac{\beta^2\bar{M}}{(2\beta-1)n\mu_x^2(k+K)}+\frac{\mathcal{O}\left(\frac{1}{\sqrt{n}(1-\rho_w)}\right)}{(k+K)^{1.5}}+\frac{\mathcal{O}\left(\frac{1}{(1-\rho_w)^2}\right)}{(k+K)^2}\label{simresult}.
\end{equation}
It is noticed that   the  algorithm  converges faster for better network connectivity (i.e., smaller $\rho_w$).  For example, a fully connected graph is the most efficient connection topology since $\rho_w=0$. In contrast, it holds $1-\rho_w\rightarrow 0$ as $n\rightarrow \infty$ for the cycle graph, which indicates that the  algorithm  will converge very slowly for large-scale cycle graphs. The following table taken from \cite[Chapter 4]{FB-LNS} characterizes the  relation between network size $n$ and the spectral gap.
Considering plugging the order concerning $n$ from the table into relation (\ref{simresult}),
we may obtain the quantitive influence of the network size on the convergence rate.
\renewcommand\arraystretch{1.5}
\begin{table}[tbhp]
	\footnotesize
	\vspace{-5pt}
	\caption{Relation between the network size $n$ and the spectral gap $1-\rho_w$ }
	\vspace{-5pt}
	\begin{center}
		\begin{tabular}{|c|c|c|c|}
			\hline
			Network Topology & Spectral Gap$(1-\rho_w)$ & Network Topology  & Spectral Gap$(1-\rho_w)$\\
			\hline
			Path Graph & $\mathcal{O}(\frac{1}{n^2})$ &2D-mesh Graph & $\mathcal{O}(\frac{1}{n})$\\
			\hline
			Cycle Graph & $\mathcal{O}(\frac{1}{n^2})$ & Complete Graph & $1$\\
			\hline
		\end{tabular}
	\end{center}
	\label{table}
\end{table}
\vspace{-15pt}
\end{remark}

There are other factors such as the strong convexity and Lipschitz smoothness parameters, as well as the variance of the stochastic gradient, all of which can also affect the convergence rate.
We will not include a quantitative analysis of these factors since the big $\mathcal{O}$ constant in the convergence rate is already quite complex and we often use the relation like $\mu_x\leq L_x$ for simplicity. While some intuitive property can be naturally obtained from (\ref{simresult}): the larger convexity and Lipschitz smoothness parameters can lead to the faster rate;  the higher variance  of stochastic  gradient descent leads to a lower convergence rate since term $\bar{M}$ defined by \eqref{barM} gets bigger.

\subsection{Transient Time}
In this subsection, we will establish the transient iteration needed for the CDSA algorithm to reach its dominant rate.

Firstly, we recall the convergence rate from  \cite[Theorem 2]{pu2021sharp} for the centralized stochastic   gradient descent,
\begin{equation}
\mathbb{E}\left[\left\|x(k)-x_*\right\|^2\right] \leq \frac{\beta^2 \bar{M}}{(2 \beta-1) n \mu^2 k}+\mathcal{O}\left(\frac{1}{n}\right) \frac{1}{k^2}.
\end{equation}
Comparing it to (\ref{simresult}), we may conclude  that our distributed algorithm converges to the optimal solution  at a comparable rate to the centralized algorithm, since they are both of the same order $\mathcal{O}(\frac{1}{k})$. Besides, our work demonstrates
that the network connectivity $\rho_w$  does not influence the term $\mathcal{O}(\frac{1}{k})$, it only appears in higher-order terms $\mathcal{O}(\frac{1}{k^{1.5}})$ and $\mathcal{O}(\frac{1}{k^{2}})$. Though our distributed algorithm asymptotically  reaches  the same order of convergence rate as that of the centralized algorithm, it's unclear how many iterations it takes to reach the dominate order $\mathcal{O}(\frac{1}{k})$ since there are two extra error terms $\mathcal{O}(\frac{1}{k^{1.5}})$ and $\mathcal{O}(\frac{1}{k^{2}})$  induced   by   averaging consensus.  We refer to the number of iterations before distributed stochastic approximation method reaches its dominant rate as \textbf{transient iterations}, i.e., when iteration $k$ is relatively small, the terms other than $n$ and $k$ still dominate the convergence rate\cite[Section 2]{nips_exp}.
The next theorem state the iterations needed for Algorithm \ref{alg:CDSA}  to reach its dominant rate.
\begin{theory}\label{thm2}
Let Assumption \ref{assump. func_pro}$\sim$\ref{assump. graph} hold, and set stepsize as (\ref{5.29}), where $\beta>2$. It takes $K_{T}=\mathcal{O}\big(\frac{n}{(1-\rho_w)^2}\big)$ iteration counts for \cref{alg:CDSA} to reach the asymptotic rate of convergence, i.e. when $k\geq K_T$, we have $\frac{1}{n}\sum\nolimits_{i=1}^{n}\mathbb{E}[||x_i(k)-x_*||^2]\leq \frac{\beta^2\bar{M}}{(2\beta-1)n\mu_x^2k}\mathcal{O}(1)$.
\end{theory}
\begin{proof}
Recalling relation in \cref{simresult}, we see that for any $
k\geq \mathcal{O}\left(\frac{n}{(1-\rho_w)^2}\right)$,
\begin{align*}	
	\frac{\beta^2\bar{M}}{(2\beta-1)n\mu_x^2(k+K)}&\geq\mathcal{O}(\frac{1}{\sqrt{n}(1-\rho_w)})
	\frac{1}{(k+K)^{1.5}},\\
	\frac{\beta^2\bar{M}}{(2\beta-1)n\mu_x^2(k+K)}&\geq\mathcal{O}\left(\frac{1}{(1-\rho_w)^2}\right)\frac{1}{(k+K)^2}.
\end{align*}
\end{proof}
\vspace{-15pt}
\section{Experiments}
\label{sec:experiments}
In this section, we will provide numerical examples to verify our theoretical findings, and
carry out  experiments by Bluefog\footnote{https://github.com/Bluefog-Lib/bluefog}. It is a python library that can be connected to the NVIDIA Collective Communications Library (NCCL) for multi-GPU computing or Message Passing Interface (MPI) library for multi-CPU computing\cite{bluefog}, i.e., each agent in our distributed experiment scenario is CPU.
\subsection{Ridge Regression}
Consider the following ridge-distributed regression problem with an unknown regularization  parameter $\theta_*$,
\[
\mathcal{C}_x(\theta_*): \min _{x \in \mathbb{R}^p} \sum\nolimits_{i=1}^n \mathbb{E}_{u_i, v_i}\left[\left(u_i^T x-v_i\right)^2+\theta_*|| x||^2\right],
\]
where $\theta_*$ can be obtained by  the distributed  learning  problem below,
\[
\mathcal{L}_{\theta}: \theta_*=\operatorname{argmin} \sum\nolimits_{i=1}^n\left(\theta-\alpha_i\right)^2.
\]
Specially,  for agent $i\in \mathcal{N}\triangleq \{1,\cdots, n\}$, its local objective functions are specified as
\[
f_i(x;\theta )= \min_x\mathbb{E}_{u_i, v_i}\left[\left(u_i^T x-v_i\right)^2+\theta|| x||^2\right], h_i(\theta)=\min_{\theta} (\theta-\alpha_i)^2.
\]
Here  $(u_i,v_i)$  are data sample collected by each agent $i$, where $u_i\in \mathbb{R}^p$ are the sample features, while  $v_i\in \mathbb{R}$ represent the observed outputs.

{\bf Parameter settings.} Set $ p=5$ and suppose that for all $i\in\mathcal{N}$,
each component  of $u_i\in \mathbb{R}^p $  is an independent identical distribution in $U(-0.5,0.5)$, and $v_i $ is drawn according to $v_i=u_i^T\widetilde{x}_i+\epsilon_i$, where  $\epsilon_i$ is an  gaussian random variable specified by $N(0,0.01)$,  and  $\widetilde{x}_i=(1~3~ 5~ 4 ~9)$ is a predefined parameter. Set $\alpha_i=0.01\times i$. It can be easily calculated that the   optimal solutions are
$	\theta_* =0.005 (n+1)$, and $x_*=\left[\sum\nolimits_{i=1}^{n}\mathbb{E}_{u_i}(u_iu_i^T)
+n\theta_*\mathbf{I}\right]\mathbb{E}_{u_i}(u_iu_i^T)
=\frac{1}{12}(\frac{1}{12}+\theta_*)^{-1}\frac{1}{n}\sum\nolimits_{i=1}^{n}\widetilde{x}_i$.

We compare the performance of Algorithm \ref{alg:CDSA} under the path graph and complete graph topology with different network size $n$.  In light of the results in \cref{table} of the path graph and  complete graph, convergence rate estimation can be reformulated.
\begin{align}
Path:	\frac{1}{n}\sum\nolimits_{i=1}^{n}\mathbb{E}[||x_i(k)-x_*||^2]\leq& \frac{\beta^2\bar{M}}{(2\beta-1)n\mu_x^2(k+K)}+\frac{\mathcal{O}(n^{3/2})}{(k+K)^{1.5}}+\frac{\mathcal{O}(n^2)}{(k+K)^2}\label{path},\\
Complete: \frac{1}{n}\sum\nolimits_{i=1}^{n}\mathbb{E}[||x_i(k)-x_*||^2]\leq& \frac{\beta^2\bar{M}}{(2\beta-1)n\mu_x^2(k+K)}+\frac{\mathcal{O}(1/\sqrt{n})}{(k+K)^{1.5}}+\frac{1}{(k+K)^2}\label{complete}.
\end{align}

We run Algorithm \ref{alg:CDSA}, where the initial values are set as $(x_i(0),\theta_i(0))=(\pmb{0}_5,1) \forall i$, and
the weighted adjacency matrix of the communication network is built according to the Metropolis-Hastings rule \cite{survey2}. According to (\ref{5.29}), we choose the stepsizes as $\alpha_k=\gamma_k =\frac{20}{k+20}$ for any $k\geq 0$.
\begin{figure}[h]
\begin{minipage}{8.2cm}
	\centering
	\includegraphics[width=7.5cm]{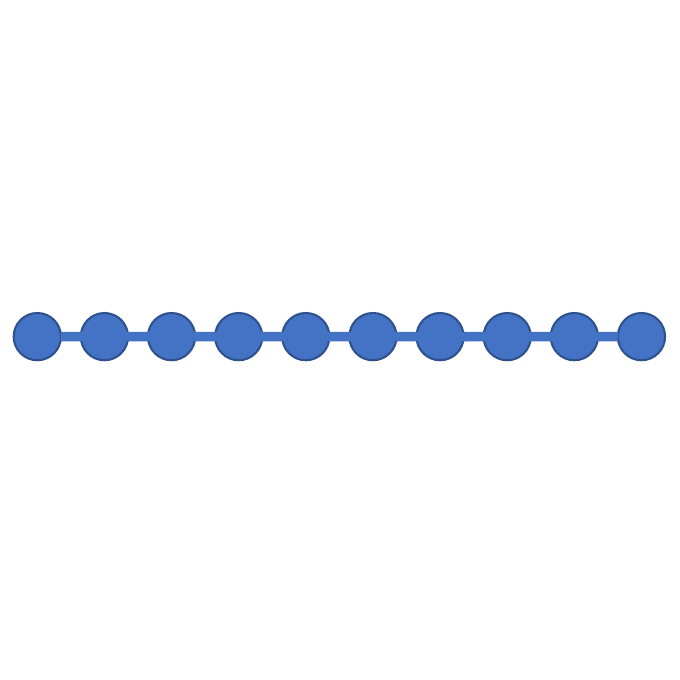}
	\centering{\fontsize{9pt}{25pt}\mdseries(a.1)n=10 path graph topology}
\end{minipage}
\begin{minipage}{8cm}
	\centering
	\includegraphics[width=7.5cm]{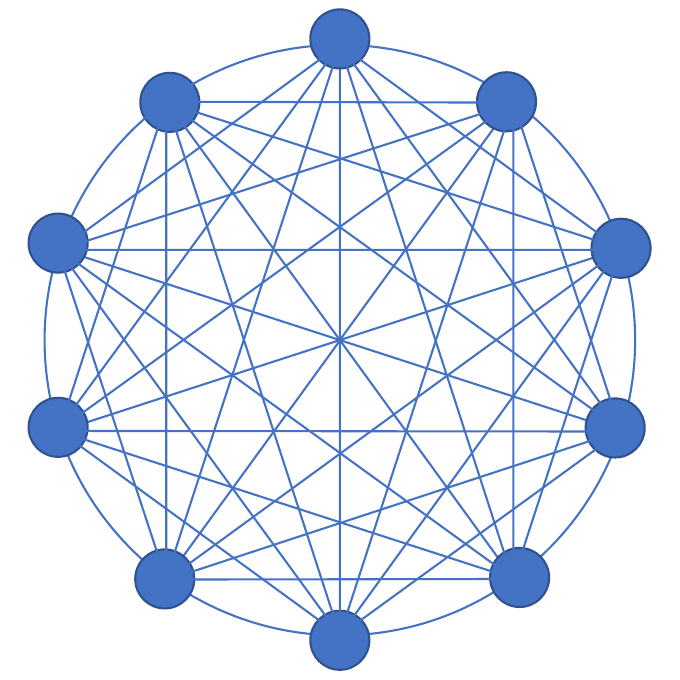}
	\centering{\fontsize{9pt}{25pt}\mdseries(b.1)n=10 complete graph topology}
\end{minipage}
\\
\begin{minipage}{8cm}
	\centering
	\includegraphics[width=7.8cm]{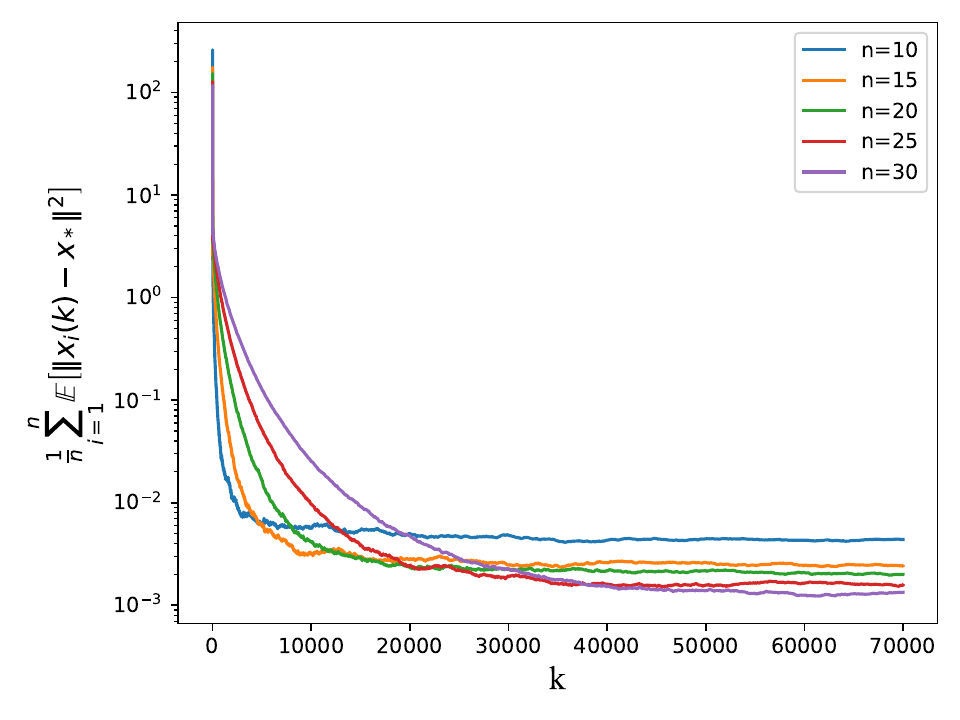}
	\centering{\fontsize{9pt}{25pt}\mdseries(a.2) The performance of path graph}
\end{minipage}
\begin{minipage}{8cm}
	\centering
	\includegraphics[width=7.8cm]{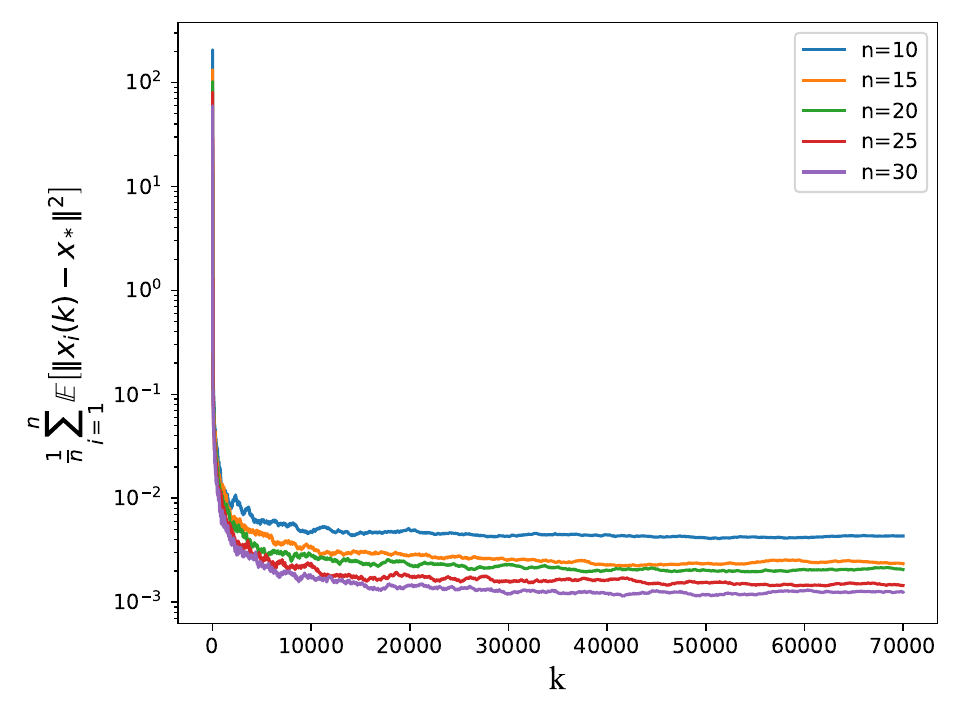}
	\centering{\fontsize{9pt}{25pt}\mdseries(b.2) The performance of complete graph}
\end{minipage}

\caption{The performance of CDSA between path graph and complete graph topology. The results are averaged over 200 Monte Carlo sampling.}
\label{fig2}
\end{figure}

We demonstrate the empirical results in Fig.~2, where the empirical  mean-squared error $\frac{1}{n}\sum\nolimits_{i=1}^{n}\mathbb{E}[||x_i(k)-x_*||^2]$  is calculated  by averaging through 200 sample paths.  We can see from the Subfigure (a.2) that for  the path graph,  when the iterate $k$ is small,    the larger network size $n$  will lead  to the higher  mean-squared error $\frac{1}{n}\sum\nolimits_{i=1}^{n}\mathbb{E}[||x_i(k)-x_*||^2]$.  However, with the increase of $k$,  we observe a  phase transition that a larger  network size $n$ will lead to a smaller  mean-squared error $\frac{1}{n}\sum\nolimits_{i=1}^{n}\mathbb{E}[||x_i(k)-x_*||^2]$ (namely faster convergence rate). This phenomenon matches  
the theoretical result (\ref{path}): when $k$ is small, the main factor influencing the convergence rate is the second and third term concerning   the network size $n$ via the distributed consensus protocol, while when $k$ is large, the first term inherited  from centralized stochastic gradient descent dominates the convergence rate.

Compared it  to the empirical performance  of the complete graph shown in subfigure (b.2), we can find that from the beginning to the end, a larger network  $n$ generates   smaller errors, which  also matches (\ref{complete}).  
\subsection{Logistic Regression}
We further consider  convex but not strongly convex problem, and  use logistic regression to demonstrate that our algorithm can also leads to asymptotic convergence.

Consider the binary classification via logistic regression with unknown regularization parameter $\theta_*$,
\[
\mathcal{C}_{\eta}(\theta_*):~\min_{\eta}\sum\nolimits_{i=1}^{n}\sum\nolimits_{j=1}^{m_i}\ln \left(1+e^{-\eta^Tx_{ij}l_{ij}}\right)+\frac{\theta_*}{2}||\eta||^2,
\]
where $\theta_*$ can be obtained by a distributed parameter learning problem as follow,
\[
\mathcal{L}_{\theta}:~\theta_*=\operatorname{argmin} \sum\nolimits_{i=1}^n\left(\theta-\alpha_i\right)^2.
\]
As for agent $i$, its  its own local computational problem and parameter learning problem are as follows.
\[
f_i(\eta;\theta )=\min_{\eta} \sum\nolimits_{j=1}^{m_i}\ln \left(1+e^{-\eta^Tx_{ij}l_{ij}}\right)+\frac{\theta_*}{2n}||\eta||^2,~ h_i(\theta)=\min_{\theta} (\theta-\alpha_i)^2.
\]

In this scenario, we set $\alpha_i=0.01\times i$ and let each agent $i\in\mathcal{N}$ possess  dataset $\mathcal{D}_i\triangleq\{(x_{ij},l_{ij}): j=1,\cdots m_i\}$, where $x_{ij}$ represents a three-dimensional sample feature where the first dimension is $1$ and the other two dimension are selected from $N((1, 0)^T, \pmb{I})$ or $N((0, 1)^T, \pmb{I})$, while $l_{ij}$ is the related sample label $1$ or $-1$ respectively. Suppose that every agent holds a number of positive samples and negative samples which only accessible to itself.  
\begin{figure}[htbp]
\centering
\includegraphics[width=10cm]{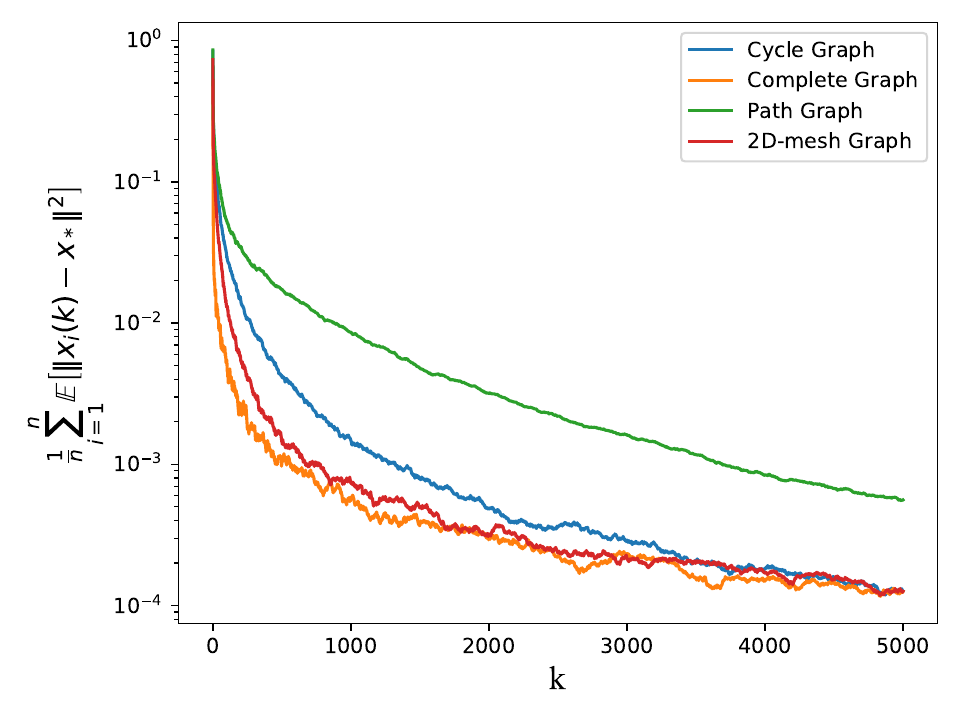}
\caption{The performance of CDSA of $25$ agents under four topologies in \cref{table}  for binary classification via logistic regression. The results are averaged over 200 Monte Carlo sampling.}
\label{figlog}				
\end{figure}

We now  compare the empirical performance of Algorithm \ref{alg:CDSA}  under  four  classes of   graph topologies, path graph, cycle graph, 2D-mesh graph, and complete graph. We set $n=25$ and  run Algorithm \ref{alg:CDSA} with initial values $(\eta_i(0),\theta_i(0))=\pmb{0}_4$ for all $i\in \mathcal{N}$, where the stepsize and weighted adjacency matrix are set the same as Ridge Regression. The empirical results are shown in
\cref{figlog}, which  shows that the complete graph has best performance, 2D-mesh graph  has the second-best performance, while the path graph displays the worst performance. These empirical findings match that listed in  \cref{table},  where the 2D-mesh graph has a  larger spectral gap  than  the path graph and cycle path,  hence leads to a lower mean-squared error.

%
\section{Conclusions}
\label{sec:conclusions}
In this work, we consider the distributed optimization problem $\min_x \frac{1}{n}\sum\nolimits_{i=1}^{n}f_i(x;\theta^*)$ with the unknown parameter $\theta^*$  collaboratively  solved by a distributed parameter learning problem $\min_{\theta} \frac{1}{n}\sum\nolimits_{i=1}^{n}h_i( \theta )$. Each agent only has access to its local computational problem $f_i(x,\theta)$ and its parameter learning problem $h_i(\theta)$.
We  propose a  coupled distributed stochastic approximation algorithm for  resolving this special distributed optimization, where  agents can  exchange information about decision variables $x$ and learning parameter $\theta$ with  neighbors  over a connected network. We quantitatively characterize the factors that influence  the  rate of convergence, and validates that the 
algorithm asymptotically achieves the optimal network-independent convergence rate compared to the centralized algorithm scheme.
In addition, we analyze the transient time $K_T$,  and show that when the iterate $k\geq K_T$, the dominate factor  influencing the convergence rate  is  related to stochastic gradient descent,
while for small $k<K_T$, the main factor   influencing the convergence rate  originates  from the distributed average consensus method. Future work will consider more general problems under weakened assumptions. It is of interests to explore the accelerated algorithm to obtain a faster convergence rate.
\bibliographystyle{elsarticle-num}
\bibliography{references}
\appendix
\section{Proof of Lemma \ref{55.3}}
\label{appA}
\begin{proof}
By using Assumption \ref{assump. moment}, we obtain that
\begin{align}
	\mathbb{E}[&||\bar{g}(\pmb{\mathbf{x}}(k),\pmb{\mathbf{\theta}}(k),\pmb{\mathbf{\xi}}(k))-\bar{\nabla}_xF(\pmb{\mathbf{x}}(k),\pmb{\mathbf{\theta}}(k))||^2|\mathcal{F}(k)] \notag\\
	&=\mathbb{E}\left[ \Big\|\frac{1}{n}\sum\nolimits_{i=1}^{n}g_i(x_i(k),\theta_i(k),\xi_i(k))-\frac{1}{n}\sum\nolimits_{i=1}^{n}\nabla_xf_i(x_i(k),\theta_i(k))\Big\|^2|\mathcal{F}(k)\right]\notag \\
	&=\frac{1}{n^2}\sum\nolimits_{i=1}^{n}\mathbb{E}\left[||g_i(x_i(k),\theta_i(k),\xi_i(k))-\nabla_xf_i(x_i(k),\theta_i(k))||^2|\mathcal{F}(k)\right]\notag\\
	&\leq\frac{1}{n^2}\sum\nolimits_{i=1}^{n}\left(\sigma_x^2+M_x||\nabla_xf_i(x_i(k),\theta_i(k))||^2\right)\leq\frac{\sigma_x^2}{n}+\frac{M_x\sum\nolimits_{i=1}^{n}||\nabla_xf_i(x_i(k),\theta_i(k))||^2}{n^2} \label{5.3},
\end{align}
where the second equality use the fact that $\xi_i,\forall i$ are independent random variables.
By recalling assumption \ref{assump. func_pro}, we achieve
\begin{align}
	&	||\nabla_xf_i(x_i(k),\theta_i(k))||^2
	=||\nabla_xf_i(x_i(k),\theta_i(k))-\nabla_xf_i(x_*,\theta_i(k))\notag\\
	&\qquad+\nabla_xf_i(x_*,\theta_i(k)) -\nabla_xf_i(x_*,\theta_*)+\nabla_xf_i(x_*,\theta_*)||^2\notag\\
	&\leq 3||\nabla_xf_i(x_i(k),\theta_i(k))-\nabla_xf_i(x_*,\theta_i(k))||^2+3||\nabla_xf_i(x_*,\theta_i(k))\notag\\
	&\qquad -\nabla_xf_i(x_*,\theta_*)||^2+3||\nabla_xf_i(x_*,\theta_*)||^2\notag\\
	&\leq3L_x^2||x_i(k)-x_*||^2+3L_{\theta}^2||\theta_i(k)-\theta_*||^2+3||\nabla_xf_i(x_*,\theta_*)||^2\label{5.4}.
\end{align}
Combining (\ref{5.4}) and  (\ref{5.3}) yields the result (\ref{5.1}).
\end{proof}
\section{Proof of Lemma \ref{55.4}}
\label{appB}
\begin{proof}
By recalling the definition of $\bar{x}(k), \bar{\theta}(k)$ and $\bar{\nabla}_xF(\pmb{\mathbf{x}}(k),\pmb{\mathbf{\theta}}(k))$  in (\ref{bar_xtheta}) and (\ref{bar_F}), using Assumption \ref{assump. func_pro}, we have
\begin{align*}
	&||\nabla_xf(\bar{x}(k),\bar{\theta}(k))-\bar{\nabla}_xF(\pmb{\mathbf{x}}(k),\pmb{\mathbf{\theta}}(k))||\\
	=&||\frac{1}{n}\sum\nolimits_{i=1}^{n}\nabla_xf_i(\bar{x}(k),\bar{\theta}(k))-\frac{1}{n}\sum\nolimits_{i=1}^{n}\nabla_xf_i(x_i(k),\theta_i(k))||\\
	\leq&\frac{1}{n}\sum\nolimits_{i=1}^{n}||\nabla_xf_i(\bar{x}(k),\bar{\theta}(k))-\nabla_xf_i(x_i(k),\theta_i(k))||\\
	=&\frac{1}{n}\sum\nolimits_{i=1}^{n}||\nabla_xf_i(\bar{x}(k),\bar{\theta}(k))-\nabla_xf_i(x_i(k),\bar{\theta}(k))+\nabla_xf_i(x_i(k),\bar{\theta}(k))-\nabla_xf_i(x_i(k),\theta_i(k))||\\
	\leq&\frac{1}{n}\sum_{i=1}^{n}\left[||\nabla_xf_i(\bar{x}(k),\bar{\theta}(k))-\nabla_xf_i(x_i(k),\bar{\theta}(k))||+||\nabla_xf_i(x_i(k),\bar{\theta}(k))-\nabla_xf_i(x_i(k),\theta_i(k))||\right]\\
	\leq&\frac{1}{n}\left(L_x\sum\nolimits_{i=1}^{n}||\bar{x}(k)-x_i(k)||+L_{\theta}\sum\nolimits_{i=1}^{n}||\bar{\theta}(k)-\theta_i(k)||\right)\\
	\leq&\frac{L_x}{\sqrt{n}}||\pmb{\mathbf{x}}(k)-\pmb{1}\bar{x}(k)^T||+\frac{L_{\theta}}{\sqrt{n}}||\pmb{\mathbf{\theta}}(k)-\pmb{1}\bar{\theta}(k)^T||,
\end{align*}
\vspace{-5pt}
where the last relation follows from Cauchy-Schwarz inequality.
\end{proof}

%
\vspace{-10pt}
\section{Proof of Lemma \ref{55.5}}
\label{appD}
\begin{proof}
According to the definitions of $\bar{x}(k)$ and $\bar{g}(\pmb{\mathbf{x}}(k),\pmb{\mathbf{\theta}}(k),\pmb{\mathbf{\xi}}(k))$ in (\ref{bar_xtheta}) and (\ref{bar_g}), together with $\sum\nolimits_{i=1}^{n}w_{ij}=1$ form Assumption \ref{assump. graph}, we have
\setlength\abovedisplayskip{3pt}
\setlength\belowdisplayskip{3pt}
\begin{align}
	&	\bar{x}(k+1)=\frac{1}{n}\sum\nolimits_{i=1}^{n}\left(\sum\nolimits_{j=1}^{n}w_{ij}\left(x_j(k)-\alpha_kg_j(x_j(k),\theta_j(k),\xi_j(k))\right)\right)\notag\\
	&=\frac{1}{n}\sum_{j=1}^{n}x_j(k)-\alpha_k\cdot\frac{1}{n}\sum_{j=1}^{n}g_j(x_j(k),\theta_j(k),\xi_j(k))=\bar{x}(k)-\alpha_k\bar{g}(\pmb{\mathbf{x}}(k),\pmb{\mathbf{\theta}}(k),\pmb{\mathbf{\xi}}(k))\label{recur_ave}.
\end{align}
Thus,
\begin{align*} &||\bar{x}(k+1)-x_*||^2=||\bar{x}(k)-\alpha_k\bar{g}(\pmb{\mathbf{x}}(k),\pmb{\mathbf{\theta}}(k),\pmb{\mathbf{\xi}}(k))-x_*||^2\notag\\
	&=||\bar{x}(k)-\alpha_k\bar{\nabla}_xF(\pmb{\mathbf{x}}(k),\pmb{\theta}(k))-x_*+\alpha_k\bar{\nabla}_xF(\pmb{\mathbf{x}}(k),\pmb{\theta}(k))-\alpha_k\bar{g}(\pmb{\mathbf{x}}(k),\pmb{\mathbf{\theta}}(k),\pmb{\mathbf{\xi}}(k))||^2\notag\\
	&=||\bar{x}(k)-\alpha_k\bar{\nabla}_xF(\pmb{\mathbf{x}}(k),\pmb{\theta}(k))-x_*||^2+\alpha_k^2||\bar{\nabla}_xF(\pmb{\mathbf{x}}(k),\pmb{\theta}(k))-\bar{g}(\pmb{\mathbf{x}}(k),\pmb{\mathbf{\theta}}(k),\pmb{\mathbf{\xi}}(k))||^2\notag\\
	&+2\alpha_k(\bar{x}(k)-\alpha_k\bar{\nabla}_xF(\pmb{\mathbf{x}}(k),\pmb{\theta}(k))-x_*)^T\left(\bar{\nabla}_xF(\pmb{\mathbf{x}}(k),\pmb{\theta}(k))-\bar{g}\left(\pmb{\mathbf{x}}(k),\pmb{\mathbf{\theta}}(k),\pmb{\mathbf{\xi}}(k)\right)\right).
\end{align*}
In light of Assumption \ref{assump. moment} and Lemma \ref{55.3}, by taking conditional expectation on both sides of above equation, we have
\begin{align}
	\mathbb{E}[|\bar{x}(k+1)-x_*||^2&|\mathcal{F}(k)]|\leq||\bar{x}(k)-\alpha_k\bar{\nabla}_xF(\pmb{\mathbf{x}}(k),\pmb{\theta}(k))-x_*||^2\notag\\
	&+\alpha_k^2\left(\frac{3M_xL_x^2}{n^2}||\pmb{\mathbf{x}}(k)-\pmb{1}x_*^T||^2+\frac{3M_xL_{\theta}^2}{n^2}||\pmb{\mathbf{\theta}}(k)-\pmb{1}\theta_*^T||^2+\frac{\bar{M}}{n}\right)\label{5.12}.
\end{align}
Next, we bound the first term on the right side of (\ref{5.12}).
\begin{align}
	||\bar{x}&(k)-\alpha_k\bar{\nabla}_xF(\pmb{\mathbf{x}}(k),\pmb{\theta}(k))-x_*||^2\notag\\
	&=||\bar{x}(k)-\alpha_k\nabla_xf(\bar{x}(k);\bar{\theta}(k))-x_*+\alpha_k\nabla_xf(\bar{x}(k),\bar{\theta}(k))-\alpha_k\bar{\nabla}_xF(\pmb{\mathbf{x}}(k),\pmb{\theta}(k))||^2\notag\\
	&=||\bar{x}(k)-\alpha_k\nabla_xf(\bar{x}(k),\bar{\theta}(k))-x_*||^2+\alpha_k^2||\nabla_xf(\bar{x}(k),\bar{\theta}(k))-\bar{\nabla}_xF(\pmb{\mathbf{x}}(k),\pmb{\theta}(k))||^2\notag\\
	&+2\alpha_k(\bar{x}(k)-\alpha_k\nabla_xf(\bar{x}(k),\bar{\theta}(k))-x_*)^T(\nabla_xf(\bar{x}(k),\bar{\theta}(k))-\bar{\nabla}_xF(\pmb{\mathbf{x}}(k),\pmb{\theta}(k)))\notag\\
	&\leq\underbrace{||\bar{x}(k)-\alpha_k\nabla_xf(\bar{x}(k),\bar{\theta}(k))-x_*||^2}_\text{Term 1}+\underbrace{\alpha_k^2||\nabla_xf(\bar{x}(k),\bar{\theta}(k))-\bar{\nabla}_xF(\pmb{\mathbf{x}}(k),\pmb{\theta}(k))||^2}_\text{Term 2}\notag\\
	&+\underbrace{2\alpha_k||\bar{x}(k)-\alpha_k\nabla_xf(\bar{x}(k),\bar{\theta}(k))-x_*||\times||\nabla_xf(\bar{x}(k),\bar{\theta}(k))-\bar{\nabla}_xF(\pmb{\mathbf{x}}(k),\pmb{\theta}(k))||}_\text{Term 3}\label{5.13}.
\end{align}
As for Term 1, by leveraging the fact that $\alpha_k\leq\frac{1}{L}$, Lemma \ref{iteration gradient work} indicates
\begin{equation}
	||\bar{x}(k)-\alpha_k\nabla_xf(\bar{x}(k),\bar{\theta}(k))-x_*||^2\leq(1-\alpha_k\mu_x)^2||\bar{x}(k)-x_*||^2\label{5.14}.
\end{equation}
By Lemma \ref{55.4}, Term 2 can be bounded as follow
\begin{align}
	&\alpha_k^2||\nabla_xf(\bar{x}(k),\bar{\theta}(k))-\bar{\nabla}_xF(\pmb{\mathbf{x}}(k),\pmb{\theta}(k))||^2\leq \left(\frac{\alpha_kL_x||\pmb{\mathbf{x}}(k)-\pmb{1}\bar{x}(k)^T||}{\sqrt{n}}+\frac{\alpha_kL_{\theta}||\pmb{\mathbf{\theta}}(k)-\pmb{1}\bar{\theta}(k)^T||}{\sqrt{n}}\right)^2\notag\\
	&\leq \frac{\alpha_k^2L_x^2||\pmb{\mathbf{x}}(k)-\pmb{1}\bar{x}(k)^T||^2}{n}+\frac{\alpha_k^2L_{\theta}^2||\pmb{\mathbf{\theta}}(k)-\pmb{1}\bar{\theta}(k)^T||^2}{n}+\frac{2L_xL_{\theta}\alpha_k^2}{n}||\pmb{\mathbf{x}}(k)-\pmb{1}\bar{x}(k)^T||\times||\pmb{\mathbf{\theta}}(k)-\pmb{1}\bar{\theta}(k)^T||.\label{5.15}
\end{align}
Finally, Term 3 can be bounded by invoking the same transformation approches used in Term 1 and Term 2:
\begin{align}
	&\text{Term 3}\leq2\alpha_k(1-\alpha_k\mu_x)||\bar{x}(k)-x_*||\left(\frac{L_x}{\sqrt{n}}||\pmb{\mathbf{x}}(k)-\pmb{1}\bar{x}(k)^T||\frac{L_{\theta}}{\sqrt{n}}||\pmb{\mathbf{\theta}}(k)-\pmb{1}\bar{\theta}(k)^T||\right)\notag\\
	&\leq\frac{2\alpha_kL_x(1-\alpha_k\mu_x)||\bar{x}(k)-x_*||||\pmb{\mathbf{x}}(k)-\pmb{1}\bar{x}(k)^T||}{\sqrt{n}}+\frac{2\alpha_kL_{\theta}(1-\alpha_k\mu_x)||\bar{x}(k)-x_*||||\pmb{\mathbf{\theta}}(k)-\pmb{1}\bar{\theta}(k)^T||}{\sqrt{n}}\label{5.16}.
\end{align}
In light of relation (\ref{5.13})$\sim$(\ref{5.16}), taking full expectation on both side of relation (\ref{5.12}) yields the result (\pmb{A}). Furthermore by using mean value inequality $2ab\leq a^2+b^2$, we rearrange (\ref{non_dim.recursion}) and obtain that
\begin{align}
	U_1(k+1)
	&\leq(1-\alpha_k\mu_x)^2U_1(k)+\frac{\alpha_k^2L_x^2}{n}V_1(k)+\frac{\alpha_k^2L_{\theta}^2}{n}V_2(k)+\frac{\alpha_k^2L_x^2}{n}V_1(k)+\frac{\alpha_k^2L_{\theta}^2}{n}V_2(k)\notag\\
	&+(1-\alpha_k\mu_x)^2c_1U_1(k)+\frac{\alpha_k^2L_x^2}{n}\cdot\frac{1}{c_1}V_1(k)+(1-\alpha_k\mu_x)^2c_2U_1(k)+\frac{\alpha_k^2L_\theta^2}{n}\cdot\frac{1}{c_2}V_2(k)\notag\\
	&\qquad+\alpha_k^2\left(\frac{3M_xL_x^2}{n^2}\mathbb{E}[||\pmb{\mathbf{x}}(k)-\pmb{1}x_*^T||]+\frac{3M_xL_{\theta}^2}{n^2}\mathbb{E}[||\pmb{\theta}(k)-\pmb{1}\theta_*^T||]+\frac{\bar{M}}{n}\right)\notag\\
	&\leq(1+c_1+c_2)(1-\alpha_k\mu_x)^2U_1(k)+(2+\frac{1}{c_1})\frac{\alpha_k^2L_x^2}{n}V_1(k)+(2+\frac{1}{c_2})\frac{\alpha_k^2L_{\theta}^2}{n}V_2(k)\notag\\
	&\quad+\alpha_k^2\left(\frac{3M_xL_x^2}{n^2}\mathbb{E}[||\pmb{\mathbf{x}}(k)-\pmb{1}x_*^T||]+\frac{3M_xL_{\theta}^2}{n^2}\mathbb{E}[||\pmb{\theta}(k)-\pmb{1}\theta_*^T||]+\frac{\bar{M}}{n}\right)\label{utemp},
\end{align}
where $c_1,c_2>0$. Take $c_1=c_2=\frac{3}{16}\alpha_k\mu_x$, then $c_1+c_2=\frac{3}{8}\alpha_k\mu_x$. Noticing that $\alpha_k\leq\frac{1}{3\mu_x}$, i.e. $\alpha_k\mu_x\leq\frac{1}{3}$, we have
\begin{equation}
\begin{aligned}
	&(1+c_1+c_2)(1-\alpha_k\mu_x)^2
	=1-\frac{13}{8}\alpha_k\mu_x+\frac{1}{4}\alpha_k^2\mu_x^2+\frac{3}{8}\alpha_k^3\mu_x^3
	\\& \leq 1-\frac{13}{8}\alpha_k\mu_x+\frac{1}{12}\alpha_k\mu_x+\frac{3}{8}\times\frac{1}{9}\alpha_k\mu_x =1-\frac{3}{2}\alpha_k\mu_x,
\end{aligned}
\end{equation}
and $(2+\frac{1}{c_m})\alpha_k\leq \frac{6}{\mu_x},m=1,2$. Plug them into (\ref{utemp}) yeilds the result \pmb{B}.
\end{proof}

\section{Proof of Lemma \ref{55.7}}
\label{appE}
\begin{proof}
Recalling the definition of $\mathbf{g}(\pmb{\mathbf{x}},\pmb{\theta},\pmb{\mathbf{\xi}})$ in (\ref{g_vec}) and relation $\bar{x}(k+1)=\bar{x}(k)-\alpha_k\bar{g}(\pmb{\mathbf{x}}(k),\pmb{\mathbf{\theta}}(k),\pmb{\mathbf{\xi}}(k))$ in  \cref{recur_ave}, and using (\ref{vec_xiter}), we have
\begin{align}
	\pmb{\mathbf{x}}(k+1)-\pmb{1}&\bar{x}(k+1)=W\left(\pmb{\mathbf{x}}(k)-\alpha_k\mathbf{g}\left(\pmb{\mathbf{x}}(k),\pmb{\mathbf{\theta}}(k),\pmb{\xi}(k)\right)\right)-\pmb{1}(\bar{x}(k)-\alpha_k\bar{g}\left(\pmb{\mathbf{x}}(k),\pmb{\mathbf{\theta}}(k),\pmb{\mathbf{\xi}}(k)\right))\notag\\ &=(W-\frac{\pmb{1}\pmb{1}^T}{n})\left[(\pmb{\mathbf{x}}(k)-\pmb{1}\bar{x}(k))-\alpha_k(\mathbf{g}\left(\pmb{\mathbf{x}}(k),\pmb{\mathbf{\theta}}(k),\pmb{\xi}(k)\right)-\pmb{1}\bar{g}\left(\pmb{\mathbf{x}}(k),\pmb{\mathbf{\theta}}(k),\pmb{\mathbf{\xi}}(k)\right))\right]. 	
\end{align}
Thus by Lemma \ref{iteration consensus work}, we obtain

\setlength\abovedisplayskip{1pt}
\setlength\belowdisplayskip{3pt}
\begin{align}
	||\pmb{\mathbf{x}}&(k+1)-\pmb{1}\bar{x}(k+1)||^2\notag\\
	&\leq\rho_w^2||(\pmb{\mathbf{x}}(k)-\pmb{1}\bar{x}(k))-\alpha_k(\mathbf{g}\left(\pmb{\mathbf{x}}(k),\pmb{\mathbf{\theta}}(k),\pmb{\xi}(k)\right)-\pmb{1}\bar{g}\left(\pmb{\mathbf{x}}(k),\pmb{\mathbf{\theta}}(k),\pmb{\mathbf{\xi}}(k)\right)||^2\notag\\
	&=\rho_w^2\big[||\pmb{\mathbf{x}}(k)-\pmb{1}\bar{x}(k)||^2+\underbrace{\alpha_k^2||\mathbf{g}\left(\pmb{\mathbf{x}}(k),\pmb{\mathbf{\theta}}(k),\pmb{\xi}(k)\right)-\pmb{1}\bar{g}\left(\pmb{\mathbf{x}}(k),\pmb{\mathbf{\theta}}(k),\pmb{\mathbf{\xi}}(k)\right)||^2}_\text{Term 4}\notag\\
	&\quad\underbrace{-2\alpha_k(\pmb{\mathbf{x}}(k)
		-\pmb{1}\bar{x}(k))^T(\mathbf{g}\left(\pmb{\mathbf{x}}(k),\pmb{\mathbf{\theta}}(k),\pmb{\xi}(k)\right)-\pmb{1}\bar{g}\left(\pmb{\mathbf{x}}(k),\pmb{\mathbf{\theta}}(k),\pmb{\mathbf{\xi}}(k)\right)}_\text{Term 5}\big]\label{5.22}
\end{align}

In the following, we will separately consider	\text{Term 4} and \text{Term 5}. Note that
\begin{equation}
	\|I-\boldsymbol{1}\boldsymbol{1}^T/n\|\leq 1 \label{1_leq1}.
\end{equation}
The by using  Assumptions \ref{assump. moment} (a) and \ref{assump. moment} (b), we derive
\begin{align} \mathbb{E}\left[\alpha_k^2\right.&\left.||\mathbf{g}(\pmb{\mathbf{x}}(k),\pmb{\mathbf{\theta}}(k),\pmb{\xi}(k))-\pmb{1}\bar{g}\left(\pmb{\mathbf{x}}(k),\pmb{\mathbf{\theta}}(k),\pmb{\mathbf{\xi}}(k)\right)||^2|\mathcal{F}(k)\right]\notag\\
	&=\alpha_k^2\mathbb{E}\left[||\nabla_x\boldsymbol{F}(\pmb{\mathbf{x}}(k),\pmb{\theta}(k))-\pmb{1}\bar{\nabla}_xF(\pmb{\mathbf{x}}(k),\pmb{\theta}(k))-\nabla_x\boldsymbol{F}(\pmb{\mathbf{x}}(k),\pmb{\theta}(k))\right.\notag\\
	&\left.+\pmb{1}\bar{\nabla}_xF(\pmb{\mathbf{x}}(k),\pmb{\theta}(k))+\mathbf{g}(\pmb{\mathbf{x}}(k),\pmb{\mathbf{\theta}}(k)),\pmb{\xi}(k)-\pmb{1}\bar{g}\left(\pmb{\mathbf{x}}(k),\pmb{\mathbf{\theta}}(k),\pmb{\mathbf{\xi}}(k)\right)||^2|\mathcal{F}(k)\right]\notag\\
	&=\alpha_k^2||\nabla_x\boldsymbol{F}(\pmb{\mathbf{x}}(k),\pmb{\theta}(k))-\pmb{1}\bar{\nabla}_xF(\pmb{\mathbf{x}}(k),\pmb{\theta}(k))||^2+\alpha_k^2\mathbb{E}[||\nabla_x\boldsymbol{F}(\pmb{\mathbf{x}}(k),\pmb{\theta}(k))\notag\\
	&\qquad-\mathbf{g}((\pmb{\mathbf{x}}(k),\pmb{\mathbf{\theta}}(k),\pmb{\xi}(k))-\pmb{1}(\bar{\nabla}_xF(\pmb{\mathbf{x}}(k),\pmb{\theta}(k))-\bar{g}(\pmb{\mathbf{x}}(k),\pmb{\mathbf{\theta}}(k)))||^2|\mathcal{F}(k)]\notag\\
	&\overset{\eqref{1_leq1}}{\leq}\alpha_k^2\left[||\nabla_x\boldsymbol{F}(\pmb{\mathbf{x}}(k),\pmb{\theta}(k))-\pmb{1}\bar{\nabla}_xF(\pmb{\mathbf{x}}(k),\pmb{\theta}(k))||^2\right.\notag\\		&\left.\qquad+\mathbb{E}[||\nabla_x\boldsymbol{F}(\pmb{\mathbf{x}}(k),\pmb{\theta}(k))-\mathbf{g}(\pmb{\mathbf{x}}(k),\pmb{\mathbf{\theta}}(k),\pmb{\xi}(k))||^2|\mathcal{F}(k)]\right]\notag\\		&\leq\alpha_k^2||\nabla_x\boldsymbol{F}(\pmb{\mathbf{x}}(k),\pmb{\theta}(k))-\pmb{1}\bar{\nabla}_xF(\pmb{\mathbf{x}}(k),\pmb{\theta}(k))||^2+\alpha_k^2n\sigma_x^2+\alpha_k^2M_x||\nabla_x\boldsymbol{F}(\pmb{\mathbf{x}}(k),\pmb{\theta}(k))||^2\label{5.23}.
\end{align}
Recalling Assumption \ref{assump. moment} (a), we obtain that
\begin{align} \mathbb{E}\big[&-2\alpha_k(\pmb{\mathbf{x}}(k)-\pmb{1}\bar{x}(k))^T\left(\mathbf{g}(\pmb{\mathbf{x}}(k),\pmb{\mathbf{\theta}}(k),\pmb{\xi}(k))-\pmb{1}\bar{g}\left(\pmb{\mathbf{x}}(k),\pmb{\mathbf{\theta}}(k),\pmb{\mathbf{\xi}}(k)\right)\right)|\mathcal{F}(k)\big]	\notag\\
	&=-2\alpha_k(\pmb{\mathbf{x}}(k)-\pmb{1}\bar{x}(k))^T\left(\nabla_x\boldsymbol{F}(\pmb{\mathbf{x}}(k),\pmb{\theta}(k))-\pmb{1}\bar{\nabla}_xF(\pmb{\mathbf{x}}(k),\pmb{\theta}(k))\right)\label{5.24},
\end{align}

In light of (\ref{5.22}),(\ref{5.23}) and (\ref{5.24}), we have
\begin{align}
	\frac{1}{\rho_w^2}\mathbb{E}[&||\pmb{\mathbf{x}}(k+1)-\pmb{1}\bar{x}(k+1)||^2|\mathcal{F}(k)]\notag\\
	&\leq||\pmb{\mathbf{x}}(k)-\pmb{1}\bar{x}(k)||^2+\alpha_k^2||\nabla_x\boldsymbol{F}(\pmb{\mathbf{x}}(k),\pmb{\theta}(k))-\pmb{1}\bar{\nabla}_xF(\pmb{\mathbf{x}}(k),\pmb{\theta}(k))||^2\notag\\
	&\quad+\alpha_k^2n\sigma_x^2+\alpha_k^2M_x||\nabla_x\boldsymbol{F}(\pmb{\mathbf{x}}(k),\pmb{\theta}(k))||^2\notag\\
	&\quad-2\alpha_k(\pmb{\mathbf{x}}(k)-\pmb{1}\bar{x}(k))^T(\nabla_x\boldsymbol{F}(\pmb{\mathbf{x}}(k),\pmb{\theta}(k))-\pmb{1}\bar{\nabla}_xF(\pmb{\mathbf{x}}(k),\pmb{\theta}(k)))\notag\\
	&\leq||\pmb{\mathbf{x}}(k)-\pmb{1}\bar{x}(k)||^2+\alpha_k^2||\nabla_x\boldsymbol{F}(\pmb{\mathbf{x}}(k),\pmb{\theta}(k))-\pmb{1}\bar{\nabla}_xF(\pmb{\mathbf{x}}(k),\pmb{\theta}(k))||^2\notag\\
	&\quad+\alpha_k^2n\sigma_x^2+\alpha_k^2M_x||\nabla_x\boldsymbol{F}(\pmb{\mathbf{x}}(k),\pmb{\theta}(k))||^2\notag\\
	&\quad+c_3||\pmb{\mathbf{x}}(k)-\pmb{1}\bar{x}(k)||^2+\frac{1}{c_3}\alpha_k^2||\nabla_x\boldsymbol{F}(\pmb{\mathbf{x}}(k),\pmb{\theta}(k))-\pmb{1}\bar{\nabla}_xF(\pmb{\mathbf{x}}(k),\pmb{\theta}(k))||^2\notag\\
	&\leq(1+c_3)||\pmb{\mathbf{x}}(k)-\pmb{1}\bar{x}(k)||^2+\alpha_k^2n\sigma_x^2+\alpha_k^2\left(1+M_x+\frac{1}{c_3}\right)||\nabla_x\boldsymbol{F}(\pmb{\mathbf{x}}(k),\pmb{\theta}(k))||^2\label{5.25},
\end{align}
where $c_3>0$ is arbitrary, and the last inequality also uses the porperty in (\ref{1_leq1}).

We then consider the upper bound of $||\nabla_xF(\pmb{\mathbf{x}}(k),\pmb{\theta}(k))||^2$ as follow,
\begin{align}
	&		||\nabla_x\boldsymbol{F}(\pmb{\mathbf{x}}(k),\pmb{\theta}(k))||^2 =||\nabla_x\boldsymbol{F}(\pmb{\mathbf{x}}(k),\pmb{\theta}(k))-\nabla_x\boldsymbol{F}(\pmb{1}x_*^T,\pmb{\theta}(k))\notag\\
	&\qquad\qquad+\nabla_x\boldsymbol{F}(\pmb{1}x_*^T,\pmb{\theta}(k))-\nabla_x\boldsymbol{F}(\pmb{1}x_*^T,\pmb{1}\theta_*^T)+\nabla_x\boldsymbol{F}(\pmb{1}x_*^T,\pmb{1}\theta_*^T)||^2\notag\\
	&\leq3||\nabla_x\boldsymbol{F}(\pmb{\mathbf{x}}(k),\pmb{\theta}(k))-\nabla_x\boldsymbol{F}(\pmb{1}x_*^T,\pmb{\theta}(k))||^2\notag\\
	&\quad+3||\nabla_x\boldsymbol{F}(\pmb{1}x_*^T,\pmb{\theta}(k))-\nabla_x\boldsymbol{F}(\pmb{1}x_*^T,\pmb{1}\theta_*^T)||^2+3||\nabla_x\boldsymbol{F}(\pmb{1}x_*^T,\pmb{1}\theta_*^T)||^2\notag\\
	&\leq3L_x^2||\pmb{\mathbf{x}}(k)-\pmb{1}\bar{x}(k)||^2+3L_{\theta}^2||\pmb{\theta}(k)-\pmb{1}\theta_*^T||^2+3||\nabla_x\boldsymbol{F}(\pmb{1}x_*^T,\pmb{1}\theta_*^T)||^2\label{5.26}.
\end{align}
Let $c_3=\frac{1-\rho_w^2}{2}$. Combining (\ref{5.25}) and (\ref{5.26}), we obtain that
\begin{align}
	\frac{1}{\rho_w^2}\mathbb{E}[||\pmb{\mathbf{x}}&(k+1)-\pmb{1}\bar{x}(k+1)||^2|\mathcal{F}(k)]\notag\\
	&\leq\frac{3-\rho_w^2}{2}||\pmb{\mathbf{x}}(k)-\pmb{1}\bar{x}(k)||^2+3\alpha_k^2\left(\frac{3}{1-\rho_w^2}+M_x\right)\left(L_x^2||\pmb{\mathbf{x}}(k)-\pmb{1}\bar{x}(k)||^2\right.\notag\\
	&\qquad\left.+L_{\theta}^2||\pmb{\theta}(k)-\pmb{1}\theta_*^T||^2+||\nabla_x\boldsymbol{F}(\pmb{1}x_*^T,\pmb{1}\theta_*^T)||^2\right)+\alpha_k^2n\sigma_x^2\label{5.27}.
\end{align}
Note that $\rho_w^2(\frac{3-\rho_w^2}{2})\leq\frac{3+\rho_w^2}{4}$ by $\rho_w\in(0,1)$. Then by taking full expectation on both sides of (\ref{5.27}) and multiplying $\rho_w^2$ leads to the result (\ref{v_recursion}).
\end{proof}

\section{Proof of Lemma \ref{55.8}}
\label{appC}
\begin{proof}
For any $k\geq0$, in order to bound $\mathbb{E}[||\pmb{\mathbf{x}}(k)-\pmb{1}x_*^T||^2]$, we firstly  consider bounding $\mathbb{E}[||x_i(k)-\alpha_kg_i(x_i(k),\theta_i(k),\xi_i(k))-x_*||^2]$ for all $i\in \mathcal{N}$. By using Assumption \ref{assump. func_pro} (i) and Assumption \ref{assump. moment} (c), we have
\begin{align}
	&\mathbb{E}[||x_i(k)-\alpha_kg_i(x_i(k),\theta_i(k),\xi_i(k))-x_*||^2|\mathcal{F}(k)]=||x_i(k)-x_*-\alpha_k\nabla_xf_i(x_i(k),\theta_i(k))||^2\notag\\
	&\qquad+\alpha_k^2\mathbb{E}\left[||\nabla_xf_i(x_i(k),\theta_i(k))-g_i(x_i(k),\theta_i(k),\xi_i(k))||^2|\mathcal{F}(k)\right]\notag\\
	&\leq||x_i(k)-x_*||^2-2\alpha_k\nabla_xf_i(x_i(k),\theta_i(k))^T(x_i(k)-x_*)\notag\\
	&\qquad+\alpha_k^2||\nabla_xf_i(x_i(k),\theta_i(k))||^2+\alpha_k^2(\sigma_x^2+M_x||\nabla_xf_i(x_i(k),\theta_i(k))||^2)\notag\\
	&\leq||x_i(k)-x_*||^2-2\alpha_k\mu_x||x_i(k)-x_*||^2 +2\alpha_k||\nabla_xf_i(x_*,\theta_i(k))||||x_i(k)-x_*||\notag\\ &\qquad+\alpha_k^2(1+M_x)||\nabla_xf_i(x_i(k),\theta_i(k))||^2+\alpha_k^2\sigma_x^2\label{5.32},
\end{align}

Consider the upper bound of the term $||\nabla_xf_i(x_i(k), \theta_i(k))||^2$ on the right side of above inequality. Using Assumption \ref{assump. func_pro} (i) and (ii), we have
\begin{align}
	||\nabla_xf_i(x_i(k),\theta_i(k))||^2&=||\nabla_xf_i(x_i(k),\theta_i(k))-\nabla_xf_i(x_*,\theta_i(k))+\nabla_xf_i(x_*,\theta_i(k))\notag\\
	&\qquad-\nabla_xf_i(x_*,\theta_*)+\nabla_xf_i(x_*,\theta_*)||\notag\\
	&\leq3L_x^2||x_i(k)-x_*||^2+3L_{\theta}^2||\theta_i(k)-\theta_*||^2+3||\nabla_xf_i(x_*,\theta_*)||^2\label{5.33}.
\end{align}
We can similarly obtain $||\nabla_xf_i(x_*,\theta_i(k))||^2\leq 2L_{\theta}^2||\theta_i(k)-\theta_*||^2+2||\nabla_xf_i(x_*,\theta_*)||^2$.
Combining (\ref{5.33}) and (\ref{5.32}), it produces
\begin{align}
	\mathbb{E}[||x_i(k)&-\alpha_kg_i(x_i(k),\theta_i(k),\xi_i(k))-x_*||^2|\mathcal{F}(k)]\leq||x_i(k)-x_*||^2-2\alpha_k\mu_x||x_i(k)-x_*||^2\notag\\
	&\quad+\alpha_k^2\sigma_x^2+2\alpha_k\sqrt{2L_{\theta}^2||\theta_i(k)-\theta_*||^2+2||\nabla_xf_i(x_*,\theta_*)||^2}||x_i(k)-x_*||\notag\\
	&\quad+\alpha_k^2(1+M_x)(3L_x^2||x_i(k)-x_*||^2+3L_{\theta}^2||\theta_i(k)-\theta_*||^2+3||\nabla_xf_i(x_*,\theta_*)||^2)\notag\\
	&\leq(1-2\alpha_k\mu_x+3\alpha_k^2(1+M_x)L_x^2)||x_i(k)-x_*||^2\notag\\
	&\quad+2\alpha_k\sqrt{2L_{\theta}^2||\theta_i(k)-\theta_*||^2+2||\nabla_xf_i(x_*,\theta_*)||^2}||x_i(k)-x_*||\notag\\
	&\quad+\alpha_k^2[3(1+M_x)L_{\theta}^2||\theta_i(k)-\theta_*||^2+3(1+M_x)||\nabla_xf_i(x_*,\theta_*)||^2+\sigma_x^2]\label{5.34}.
\end{align}
From the definition of $K$ in (\ref{5.30}), for all $k\geq 0$, we have $\alpha_k\leq\frac{\mu_x}{3(1+M_x)L_x^2}$. Recall the fact that $\mathbb{E}[||\theta_i(k)-\theta_*||^2]\leq\hat{\Theta}_i$ in \eqref{theta_bound}. By taking full expectation on both sides of (\ref{5.34}) and using $ \mathbb{E}[||x_i(k)-x_*||]\leq  \sqrt{\mathbb{E}[||x_i(k)-x_*||^2]}$, we have
\begin{align}
	\mathbb{E}[||x_i(k)-&\alpha_k g_i(x_i(k),\theta_i(k),\xi_i(k))-x_*||^2]\leq(1-\alpha_k\mu_x)\mathbb{E}[||x_i(k)-x_*||^2]\notag\\
	&\qquad+2\alpha_k\sqrt{2L_{\theta}^2\hat{\Theta}_i+2||\nabla_xf_i(x_*,\theta_*)||^2}\sqrt{\mathbb{E}[||x_i(k)-x_*||^2]}\notag\\
	&\qquad+\alpha_k\left[\frac{\mu_xL_{\theta}^2}{L_x^2}\hat{\Theta}_i+\frac{\mu_x}{L_x^2}||\nabla_xf_i(x_*,\theta_*)||^2+\frac{\mu_x\sigma_x^2}{3(1+M_x)L_x^2}\right]\notag\\
	&\quad\leq\mathbb{E}[||x_i(k)-x_*||^2]-\alpha_k\bigg[\mu_x\mathbb{E}||x_i(k)-x_*||^2\notag\\
	&\qquad-2\sqrt{2L_{\theta}^2\hat{\Theta}_i+2||\nabla_xf_i(x_*,\theta_*)||^2}\sqrt{\mathbb{E}[||x_i(k)-x_*||^2]}\notag\\
	&\qquad-\left(\frac{\mu_xL_{\theta}^2}{L_x^2}\hat{\Theta}_i+\frac{\mu_x}{L_x^2}||\nabla_xf_i(x_*,\theta_*)||^2+\frac{\mu_x\sigma_x^2}{3(1+M_x)L_x^2}\right)\bigg].
	\label{5.36}
\end{align}
Next, we consider the following set:
\setlength\abovedisplayskip{3pt}
\setlength\belowdisplayskip{3pt}
\begin{align}
	\mathcal{X}_i \triangleq \bigg \{q\geq0:&\mu_x q-2\sqrt{2L_{\theta}^2\hat{\Theta}_i+2||\nabla_xf_i(x_*,\theta_*)||^2}\sqrt{q}\notag\\
	&-\frac{\mu_x}{3L_x^2}\left(3L_{\theta}^2\hat{\Theta}_i+3||\nabla_xf_i(x_*,\theta_*)||^2
	+\frac{\sigma_x^2}{1+M_x}\right)\leq0\bigg\}.\label{kaset}
\end{align}
It can be seen that $\mathcal{X}_i$ is non-empty and compact. If $\mathbb{E}[||x_i(k)-x_*||^2] \notin \mathcal{X}_i$, in light of (\ref{5.36}) we known that $\mathbb{E}[||x_i(k)-\alpha_k g_i(x_i(k),\theta_i(k),\xi_i(k))-x_*||^2]\leq\mathbb{E}[||x_i(k)-x_*||^2]$.
While for $\mathbb{E}[||x_i(k)-x_*||^2] \in \mathcal{X}_i$, by using $\alpha_k\leq\frac{\mu_x}{3(1+M_x)L_x^2}$, we derive
\begin{align}
	&\mathbb{E}[||x_i(k)-\alpha_k g_i(x_i(k),\theta_i(k),\xi_i(k))-x_*||^2]\notag\\
	&\leq\max_{q\in\mathcal{X}_i}\Bigg\{q -\frac{\mu_x}{3(1+M_x)L_x^2} \Big[\mu_x q-2\sqrt{2L_{\theta}^2\hat{\Theta}_i+2||\nabla_xf_i(x_*,\theta_*)||^2}\sqrt{ q}\notag\\
	&\qquad-\frac{\mu_x}{3L_x^2}\left(3L_{\theta}^2\hat{\Theta}_i+3||\nabla_xf_i(x_*,\theta_*)||^2
	+\frac{\sigma_x^2}{1+M_x}\right)\Big] \Bigg\}\triangleq R_i\label{5.38}.
\end{align}
Based on previous arguments, we conclude that for all $k>0$,
\begin{equation}
	\mathbb{E}[||x_i(k)-\alpha_k g_i(x_i(k),\theta_i(k),\xi_i(k))-x_*||^2]\leq\max\left\{\mathbb{E}[||x_i(k)-x_*||^2],R_i\right\}\label{Emax}.
\end{equation}
In light of $W\boldsymbol{1}=\boldsymbol{1}$, by noting from (\ref{vec_xiter}) that
\begin{align}
	\|\pmb{\mathbf{x}}(k+1)-\boldsymbol{1}x_*^T\|^2&\leq\|W\|^2\|\pmb{\mathbf{x}}(k)-\alpha_k\mathbf{g}(\pmb{\mathbf{x}}(k),\pmb{\theta}(k),\pmb{\mathbf{\xi}}(k))-\boldsymbol{1}x_*^T\|^2\notag\\
	&\leq\|\pmb{\mathbf{x}}(k)-\alpha_k\mathbf{g}(\pmb{\mathbf{x}}(k),\pmb{\theta}(k),\pmb{\mathbf{\xi}}(k))-\boldsymbol{1}x_*^T\|^2\label{vecinter}.
\end{align}
This together with  (\ref{Emax}) produces
\begin{equation}
	 \mathbb{E}[||\pmb{\mathbf{x}}(k)-\pmb{1}x_*^T||^2]\leq\max\left\{\mathbb{E}[||\pmb{\mathbf{x}}(0)-\pmb{1}x_*^T||^2],\sum\nolimits_{i=1}^{n}R_i\right\}\label{5.42}
\end{equation}

In the following,  we will give an upper bound of  $R_i$. From the definition of $\mathcal{X}_i$ in (\ref{kaset}), we  know that
the right zero of the upward opening parabola is
\begin{equation}
	\begin{aligned}
		&\sqrt{q_i}=\frac{1}{2\mu_x}\Bigg[2\sqrt{2L_{\theta}^2\hat{\Theta}_i+2||\nabla_xf_i(x_*,\theta_*)||^2}\notag\\
		&+\sqrt{4(2L_{\theta}^2\hat{\Theta}_i+2||\nabla_xf_i(x_*,\theta_*)||^2)
			+\frac{4\mu_x^2}{3L_x^2}\left(3L_{\theta}^2\hat{\Theta}_i+3||\nabla_xf_i(x_*,\theta_*)||^2+\frac{\sigma_x^2}{1+M_x}\right)}\Bigg].
	\end{aligned}
\end{equation}
Then by using $\mu_x\leq L_x$, we achieve
\begin{equation}
	\begin{aligned}
		&q_i\leq\frac{1}{4\mu_x^2}\Bigg[2\times4\big(2L_{\theta}^2\hat{\Theta}_i+2||\nabla_xf_i(x_*,\theta_*)||^2\big)\notag\\
		&\quad+2\left(8L_{\theta}^2\hat{\Theta}_i+8||\nabla_xf_i(x_*,\theta_*)||^2+ 4 L_{\theta}^2\hat{\Theta}_i \right. \left. + 4||\nabla_xf_i(x_*,\theta_*)||^2 +\frac{4\mu_x^2\sigma_x^2}{3L_x^2(1+M_x)}\right)\Bigg]\notag\\
		&\leq\frac{10L_{\theta}^2\hat{\Theta}_i}{\mu_x^2}+\frac{10||\nabla_xf_i(x_*,\theta_*)||^2}{\mu_x^2}+\frac{2\sigma_x^2}{3(1+M_x)L_x^2} \triangleq q_i^*.
	\end{aligned}
\end{equation}

Thus, $\mathcal{X}_i=[0,q_i]\subset [0,q_i^*]$. Hence from (\ref{5.38}) it follows that
\begin{align}
	R_i&\leq  q_i^*-\frac{\mu_x}{3(1+M_x)L_x^2} \Bigg[\mu_x q-2\sqrt{2L_{\theta}^2\hat{\Theta}_i+2||\nabla_xf_i(x_*,\theta_*)||^2}\sqrt{ q}\notag\\
	&\qquad-\frac{\mu_x}{3L_x^2}\left(3L_{\theta}^2\hat{\Theta}_i+3||\nabla_xf_i(x_*,\theta_*)||^2+\frac{\sigma_x^2}{1+M_x}\right)\Bigg]\Bigg{|}_{q=\frac{\sqrt{2L_{\theta}^2\hat{\Theta}_i+2||\nabla_xf_i(x_*,\theta_*)||^2}}{\mu_x}}\notag\\
	&\leq \frac{10L_{\theta}^2\hat{\Theta}_i}{\mu_x^2}+\frac{10||\nabla_xf_i(x_*,\theta_*)||^2}{\mu_x^2}+\frac{2\sigma_x^2}{3(1+M_x)L_x^2}\notag\\
	&\qquad+\frac{\mu_x}{3(1+M_x)L_x^2}\left[\frac{\mu_x}{3L_x^2}\left(3L_{\theta}^2\hat{\Theta}_i+3||\nabla_xf_i(x_*,\theta_*)||^2+\frac{\sigma_x^2}{1+M_x}\right)\right.\notag\\
	&\left.\qquad+\frac{2L_{\theta}^2\hat{\Theta}_i+2||\nabla_xf_i(x_*,\theta_*||^2)}{\mu_x}\right]\notag\\
	&\leq \frac{11L_{\theta}^2\hat{\Theta}_i}{\mu_x^2}+\frac{11||\nabla_xf_i(x_*,\theta_*)||^2}{\mu_x^2}
	+\frac{7\sigma_x^2}{9(1+M_x)L_x^2}\label{5.45},
\end{align}
where the last inequality has used $\mu_x\leq L_x$.

Combing (\ref{5.45}) and (\ref{5.42}), the lemma holds.
\end{proof}











\end{document}